\documentclass[a4paper]{article}

\usepackage[utf8]{inputenc}

\usepackage{lmodern}
\usepackage{array}
\usepackage{graphicx}
\usepackage{subfig}

\usepackage{amssymb, amsfonts, amsthm, amsmath}
\usepackage{enumerate}

\usepackage[english]{babel}

\usepackage{tikz}
\usetikzlibrary{decorations}

\usetikzlibrary{decorations.pathreplacing}

\tikzset{individu/.style={draw,thick}}

\numberwithin{equation}{section}
\allowdisplaybreaks

\theoremstyle{plain}
\newtheorem{theorem}{Theorem}[section]
\newtheorem{corollary}[theorem]{Corollary}

\newtheorem{lemma}[theorem]{Lemma}
\newtheorem{proposition}[theorem]{Proposition}

\theoremstyle{definition}
\newtheorem{definition}[theorem]{Definition}

\theoremstyle{remark}
\newtheorem{remark}[theorem]{Remark}

\newcommand{\Z}{\mathbb{Z}}
\newcommand{\R}{\mathbb{R}}

\newcommand{\T}{\mathbb{T}}

\newcommand{\calS}{\mathcal{S}}
\newcommand{\calC}{\mathcal{C}}
\newcommand{\calD}{\mathcal{D}}

\newcommand{\calI}{\mathcal{I}}

\newcommand{\calK}{\mathcal{K}}
\newcommand{\calQ}{\mathcal{Q}}

\newcommand{\calH}{\mathcal{H}}
\newcommand{\calA}{\mathcal{A}}

\newcommand{\calM}{\mathcal{M}}

\DeclareMathOperator{\lcm}{lcm}

\renewcommand{\tilde}[1]{\widetilde{#1}}
\renewcommand{\epsilon}{\varepsilon}
\renewcommand{\phi}{\varphi}

\newcommand{\Addresses}{{
  \bigskip
  \footnotesize
  
  \textsc{Unit\'e de Math\'ematiques Pures et Appliqu\'ees, \'Ecole normale sup\'erieure de Lyon, 46 all\'ee d'Italie, 69364 Lyon Cedex 07, France}\par\nopagebreak
  \textit{E-mail address}: \texttt{sanjay.ramassamy@ens-lyon.fr}

}}

\title{Miquel dynamics for circle patterns}
\author{Sanjay Ramassamy}
\date{\today}

\begin{document}

\maketitle

\begin{abstract}
We study a new discrete-time dynamical system on circle patterns with the combinatorics of the square grid. This dynamics, called Miquel dynamics, relies on Miquel's six circles theorem. We provide a coordinatization of the appropriate space of circle patterns on which the dynamics acts and use it to derive local recurrence formulas. Isoradial circle patterns arise as periodic points of Miquel dynamics. Furthermore, we prove that certain signed sums of intersection angles are preserved by the dynamics. Finally, when the initial circle pattern is spatially biperiodic with a fundamental domain of size two by two, we show that the appropriately normalized motion of intersection points of circles takes place along an explicit quartic curve.
\end{abstract}

\section{Introduction}

This paper is about a conjecturally integrable discrete-time dynamical system on circle patterns.

Circle patterns are configurations of circles on a surface with combinatorially prescribed intersections. They have been actively studied since Thurston conjectured that circle packings (a variant of circle patterns where neighboring circles are tangent) provide discrete approximations of Riemann mappings between two simply connected domains~\cite{Thurston}. The conjecture was proved by Rodin and Sullivan~\cite{RS} and generalized by He and Schramm~\cite{HS1,HS2}. Circle patterns have been used to define discrete analogs of holomorphic functions such as power functions, exponential functions and logarithms~\cite{AB,Schramm,BHS}. In order to construct circle patterns, several articles have followed a variational approach, where the intersection angles between neighboring circles is fixed and the radii of the circles are the solution of a variational problem (see~\cite{BS} and the references therein). The consistency equations verified by the parameters describing a circle pattern have been shown to be integrable in many cases~\cite{BP,BH,BHS}.

Discrete integrable systems have been the subject of intense studies in recent years. The most popular one coming from geometry is probably Schwartz's pentagram map~\cite{Schwartz1}, a discrete-time dynamical system on some space of polygons. Ovsienko, Schwartz and Tabachnikov~\cite{OST1,OST2} proved its Liouville integrability by defining an invariant Poisson bracket on the space of polygons and showing that invariants of the dynamics found in~\cite{Schwartz2} Poisson commute. Soloviev~\cite{Soloviev} proved its algebro-geometric integrability by deriving its spectral curve. Glick~\cite{Glick} established a connection with cluster algebras by identifying the recurrence formulas for the pentagram as some mutations of coefficient variables.

Another discrete integrable system, this time coming from statistical mechanics, is the dimer model. Goncharov and Kenyon~\cite{GK} studied its Liouville integrability and its algebro-geometric integrability and related it to cluster algebras. It turns out that in the scaling limit, the dimer model has several features reminiscent of those of circle patterns, though a direct connection has yet to be established. For example, Cohn, Kenyon and Propp~\cite{CKP} proved that the limit shape for dimers on the square grid was the solution of some variational problem with a functional defined using Milnor's Lobachevsky function, similar to the variational problems arising for circle patterns in~\cite{BS}. Furthermore, Kenyon and Okounkov~\cite{KO} defined a conformal structure on the limit shape, which turned out to be the right framework to describe its fluctuations~\cite{Kenyon2}. Thus, it is natural to ask whether there exists an integrable system on the space of circle patterns which would resemble the dimer integrable system.

The space on which our dynamics takes place is the space of square grid circle patterns, which are maps from the vertices of the infinite square grid $\Z^2$ to $\R^2$ such that any four vertices around a face of $\Z^2$ get mapped to four concyclic points. The circles associated with these faces are colored black or white in a checkerboard fashion. Note that our square grid circle patterns differ from Schramm's ``circle patterns with the combinatorics of the square grid''~\cite{Schramm}, because we do not require neighboring circles to intersect orthogonally\footnote{Actually Schramm's circle patterns will be fixed points of our dynamics.}.

\begin{figure}[htpb]
\centering
\includegraphics[height=2in]{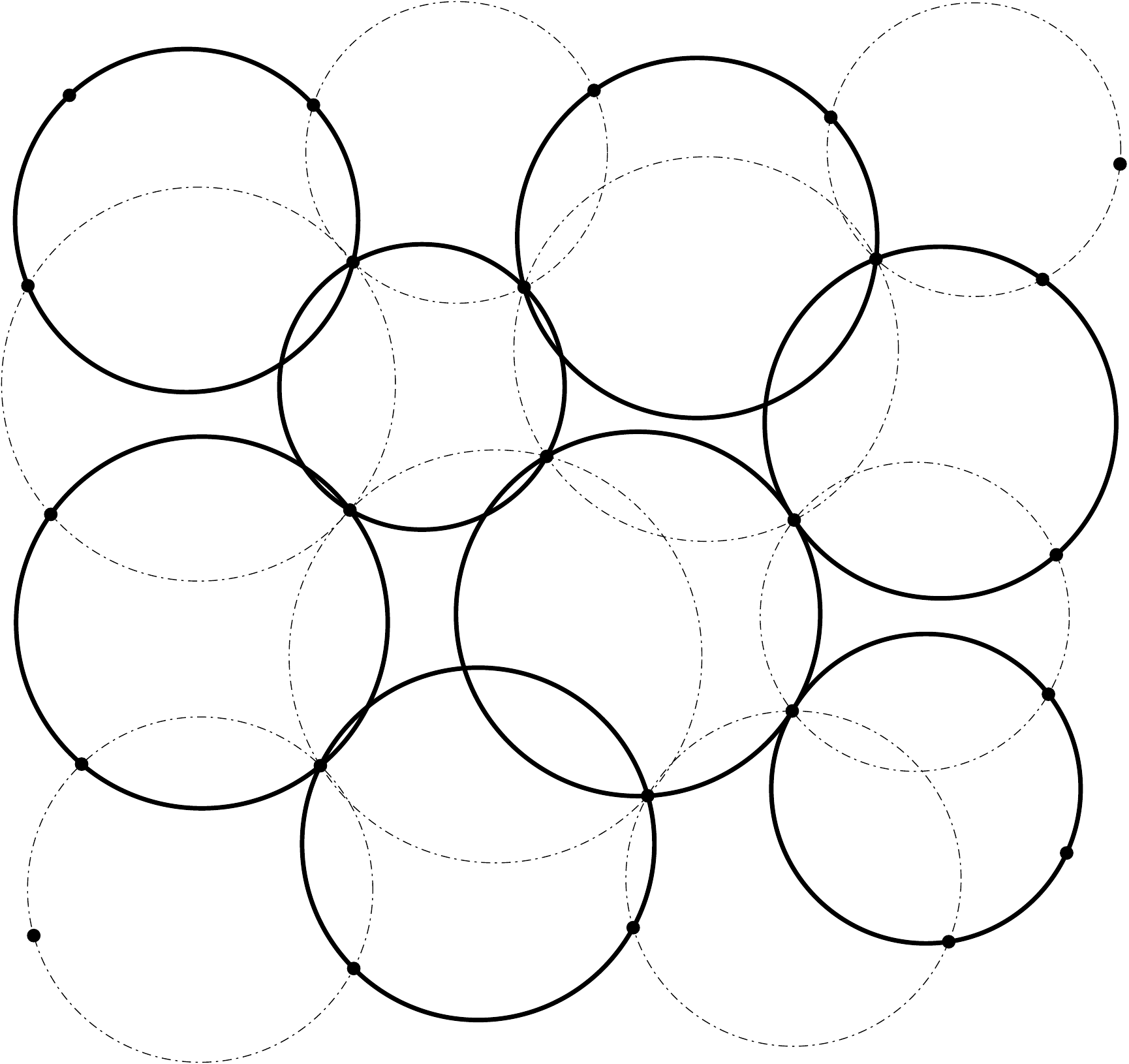}
\caption{A $4$ by $4$ portion of a square grid circle pattern. Black (resp. white) circles are represented with a bold (resp. dash-dotted) stroke.}
\label{fig:circlepattern}
\end{figure}

We now explain how to mutate a black circle. A black circle is surrounded by four white circles, corresponding the faces North, West, South and East of the black face. These four white circles have eight points of intersection, four of which lie on the black circle (see Figure~\ref{fig:blackmutation}). Miquel's six circles theorem guarantees that the other four intersection points are concyclic. The mutation of the black circle consists in erasing the original black circle and replacing it by the circle going through these four other points of intersection. Miquel dynamics is defined as the following discrete-time dynamical system. Start with a square grid circle pattern at time $0$. At odd (resp. even) times, mutate all the black (resp. white) circles simultaneously.

\begin{figure}[htpb]
\centering
\includegraphics[height=2in]{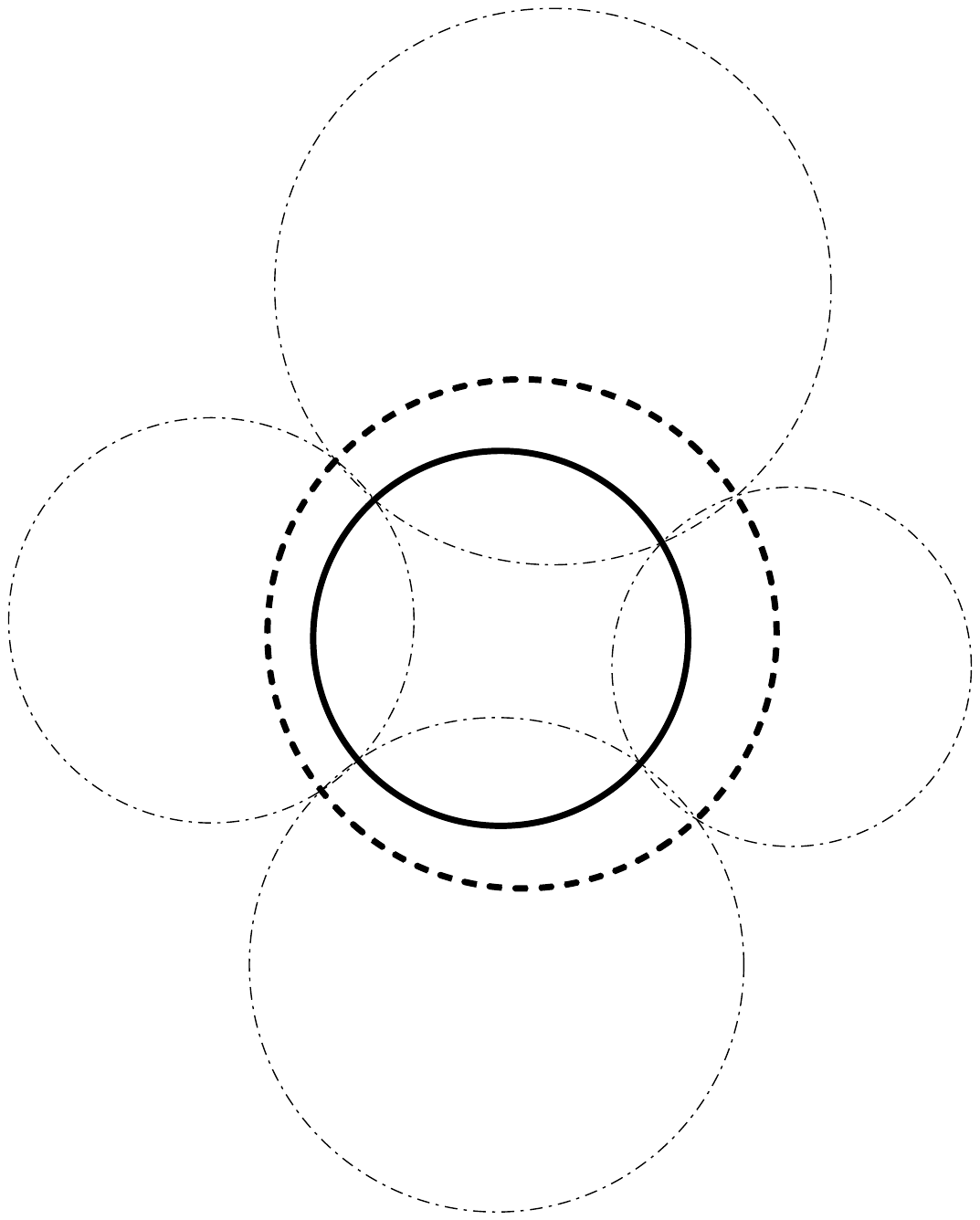}
\caption{Black mutation turns the original black circle (solid stroke) into a new black circle (dashed stroke).}
\label{fig:blackmutation}
\end{figure}

Observe that the standard embedding of $\Z^2$ is a fixed point of this dynamics, because any two diagonally neighboring circles are tangent. Miquel dynamics was first introduced by Richard Kenyon, who also conjectured its integrable nature (Liouville integrability, algebro-geometric integrability, connection to cluster algebras) when the time $0$ circle pattern is spatially biperiodic\footnote{private communication}. We illustrate this conjecture in Figure~\ref{fig:44integrability} by the plot of the relative motion of one intersection point with respect to another one when the fundamental domain is of size $4$ by $4$, in which case the motion seems to take place on the two-dimensional projection of a higher-dimensional torus.

\begin{figure}[htpb]
\centering
\includegraphics[height=2in]{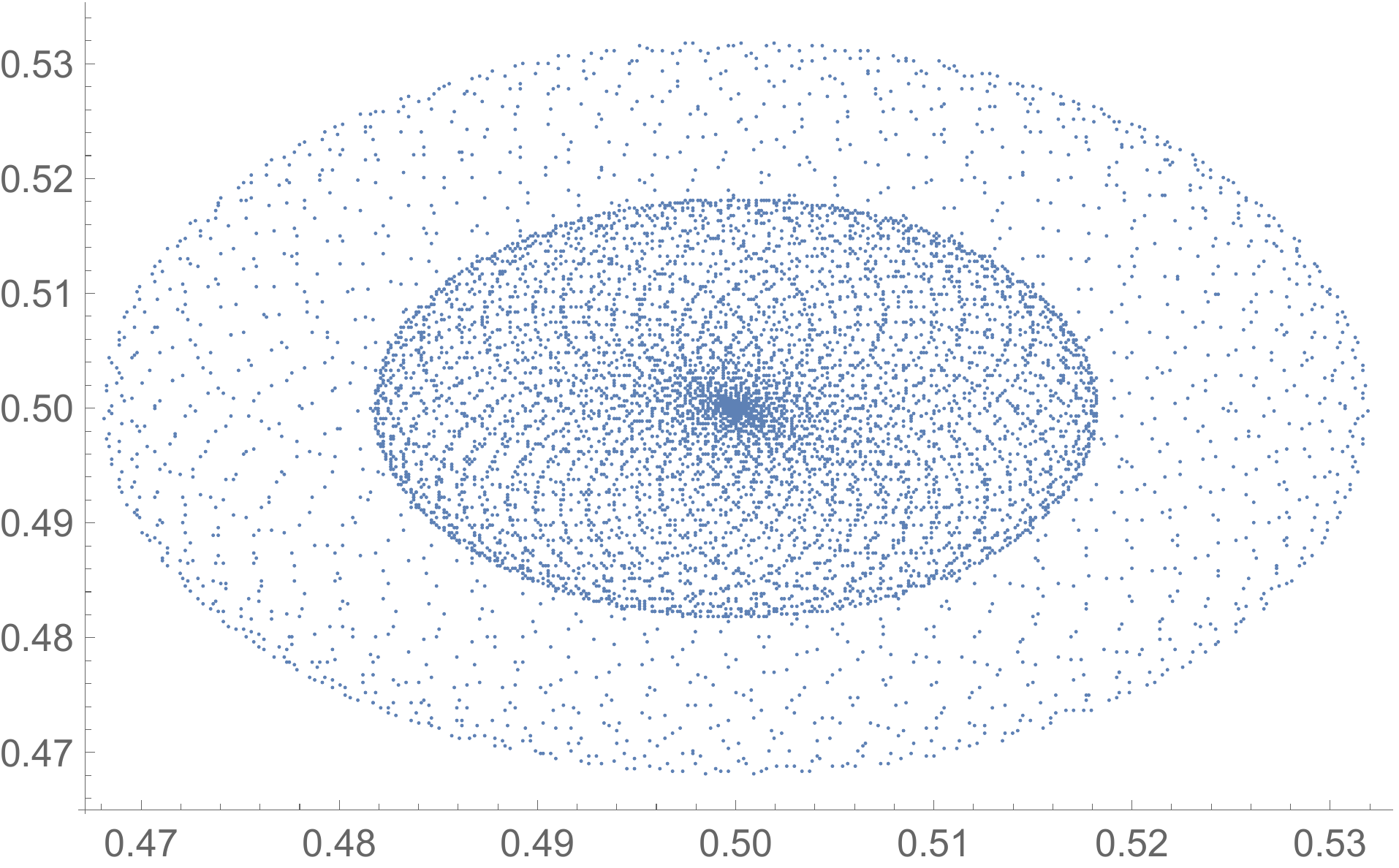}
\caption{Computer simulation of the relative motion of one intersection point with respect to another one when the fundamental domain is of size $4$ by $4$.}
\label{fig:44integrability}
\end{figure}

One can consider circle patterns with combinatorics other than the square grid and mutate a circle whenever we have a face of degree $4$. However, if we start with an arbitrary graph, we will usually not be able to iterate the dynamics indefinitely. Note that the integrability of Miquel dynamics is that of a $2+1$-dimensional system (two spatial dimensions and one time dimension) while most of the integrability properties of circle patterns that have been studied so far (for example~\cite{BP,BH,BHS}) concern the consistency equations of a two-dimensional system. Bazhanov, Mangazeev and Sergeev~\cite{BMS2} studied a dynamical system on circle patterns with the combinatorics of a graph obtained by projecting some stepped surface to a plane. The definition of their dynamics also relied on Miquel's theorem, but their local mutation scheme is different from ours : starting from three circles, they produce three new circles, while we start with five circles to produce a new circle.

We now list the main results obtained in this paper.

\begin{itemize}
\item We provide global coordinates on the space of spatially biperiodic square grid circle patterns considered up to similarity. The coordinates are the angles $\phi$ under which an edge is seen when standing at the circumcenter of a face. These angles must satisfy certain relations in order to represent a biperiodic square grid circle pattern.
\item Using these coordinates (or rather the variables $X=e^{i\phi}$) we derive local recurrence formulas describing Miquel dynamics. These formulas are rational expressions in the variables $X$ which are reminiscent of the coefficient variables mutations in cluster algebras.
\item A biperiodic square grid circle pattern can naturally be projected to a torus $\T$. For any biperiodic square grid circle pattern $S$, we define a group homomorphism
\[
\gamma_S:H_1(\T,\Z)\rightarrow\R/(2\pi\Z)
\]
corresponding to a signed sum of intersection angles of circles along a loop around the torus, and we show that $\gamma_S$ is (essentially) preserved by Miquel dynamics.
\item When the fundamental domain of the biperiodic square grid circle pattern is a two by two array of faces, we prove that the appropriately renormalized trajectory of an intersection point of circles is supported by some explicit quartic curve.
\end{itemize}

The rest of the paper is organized as follows. In Section~\ref{sec:Miquel}, we recall Miquel's theorem and provide an effective version of this theorem. In Section~\ref{sec:patterns}, we define the space of biperiodic square grid circle patterns up to similarity as well as Miquel dynamics on that space. Section~\ref{sec:coordinates} provides a coordinatization of the aforementioned space. These coordinates are then used in Section~\ref{sec:recurrenceformulas} to derive recurrence formulas for Miquel dynamics. As an application, we show in Section~\ref{sec:isoradial} that isoradial patterns are periodic points of the dynamics. Section~\ref{sec:conserved} exposes some conserved quantities of the dynamics. Finally in Section~\ref{sec:twobytwo} we study the trajectory of vertices when the fundamental domain of the biperiodic square grid circle pattern is a two by two array of faces.

\section{Miquel's theorem}
\label{sec:Miquel}

Consider four circles $\calC_1,\calC_2,\calC_3$ and $\calC_4$ on the Riemann sphere. Assume that two consecutive circles always intersect in two points : $\calC_i\cap\calC_{i+1}=\left\{A_{i,i+1},B_{i,i+1}\right\}$ for $1\leq i\leq 4$, where the indices are considered modulo $4$ (see Figure~\ref{fig:miquelthm}). Miquel's theorem~\cite{Miquel} says when we can draw two additional circles through these intersection points :

\begin{theorem}[Miquel's theorem]
The points $A_{1,2},A_{2,3},A_{3,4},A_{4,1}$ are concyclic if and only if the points $B_{1,2},B_{2,3},B_{3,4},B_{4,1}$ are concyclic.
\end{theorem}

\begin{figure}[htpb]
\centering
\includegraphics[height=3in]{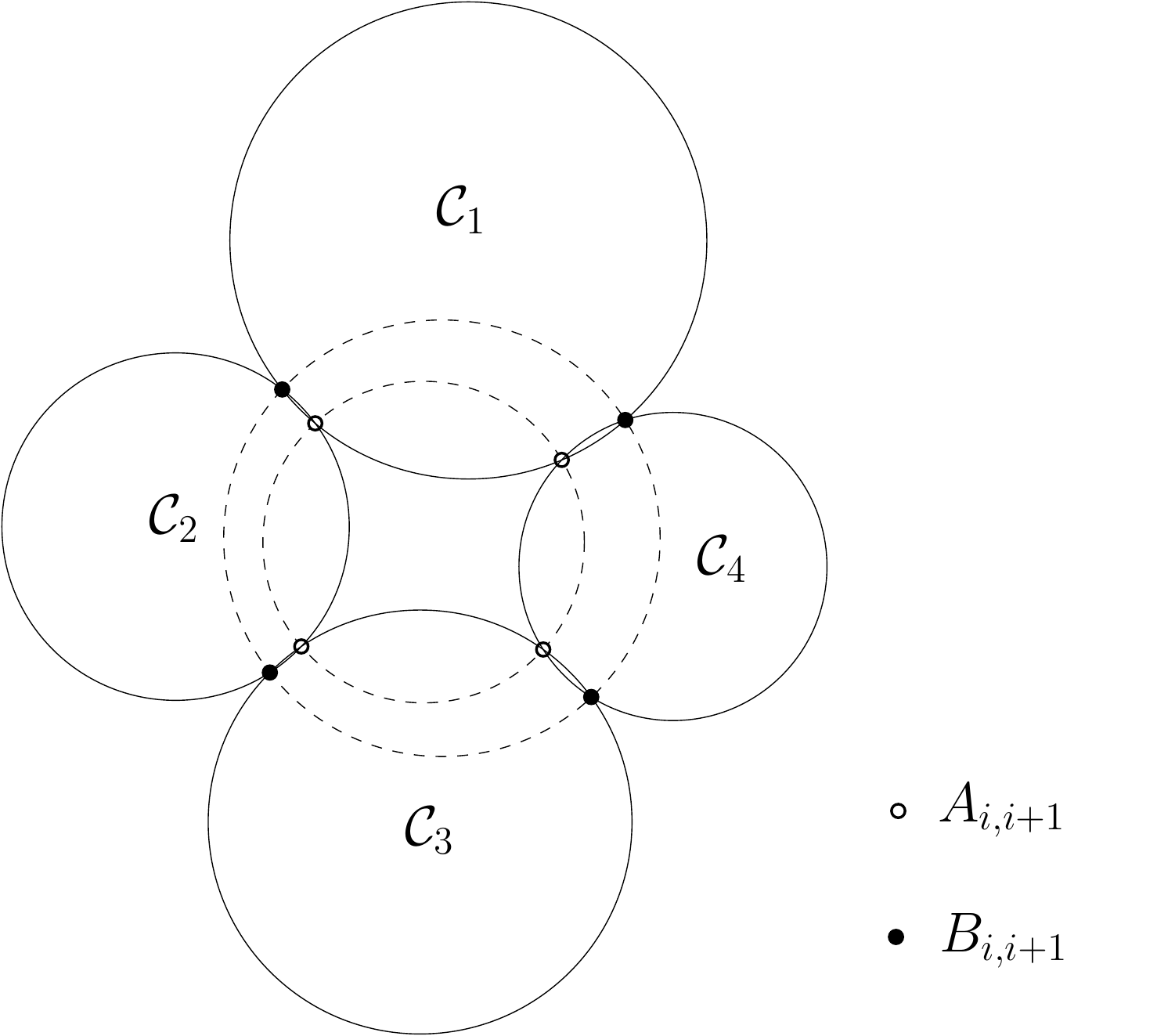}
\caption{Illustration of Miquel's theorem.}
\label{fig:miquelthm}
\end{figure}

A natural question to ask is : on what condition on the circles $\calC_i$ can we draw the two additional circles ? In order to answer it, we need to define a notion of intersection angle between circles. We can stereographically project the sphere to a plane from a point $P$ which lies on no circumcircle of any three points in $\left\{A_{1,2},A_{2,3},A_{3,4},A_{4,1},B_{1,2},B_{2,3},B_{3,4},B_{4,1}\right\}$. With such a projection, circles are guaranteed to be mapped to actual circles and not straight lines. Given three points $A,B,C$ in the plane, we denote by $\angle ABC$ the oriented angle between the vectors $\overrightarrow{BA}$ and $\overrightarrow{BC}$. In particular, $\angle CBA=2\pi-\angle ABC$. Let $O_i$ denote the center of the circle $\calC_i$. Denote by $\theta_{i,i+1}$ the \emph{exterior intersection angle} between the circles $\calC_i$ and $\calC_{i+1}$ (see Figure~\ref{fig:thetadef}) :
\[
\theta_{i,i+1}=\angle O_iA_{i,i+1}O_{i+1}.
\]
\begin{figure}[htpb]
\centering
\includegraphics[height=2in]{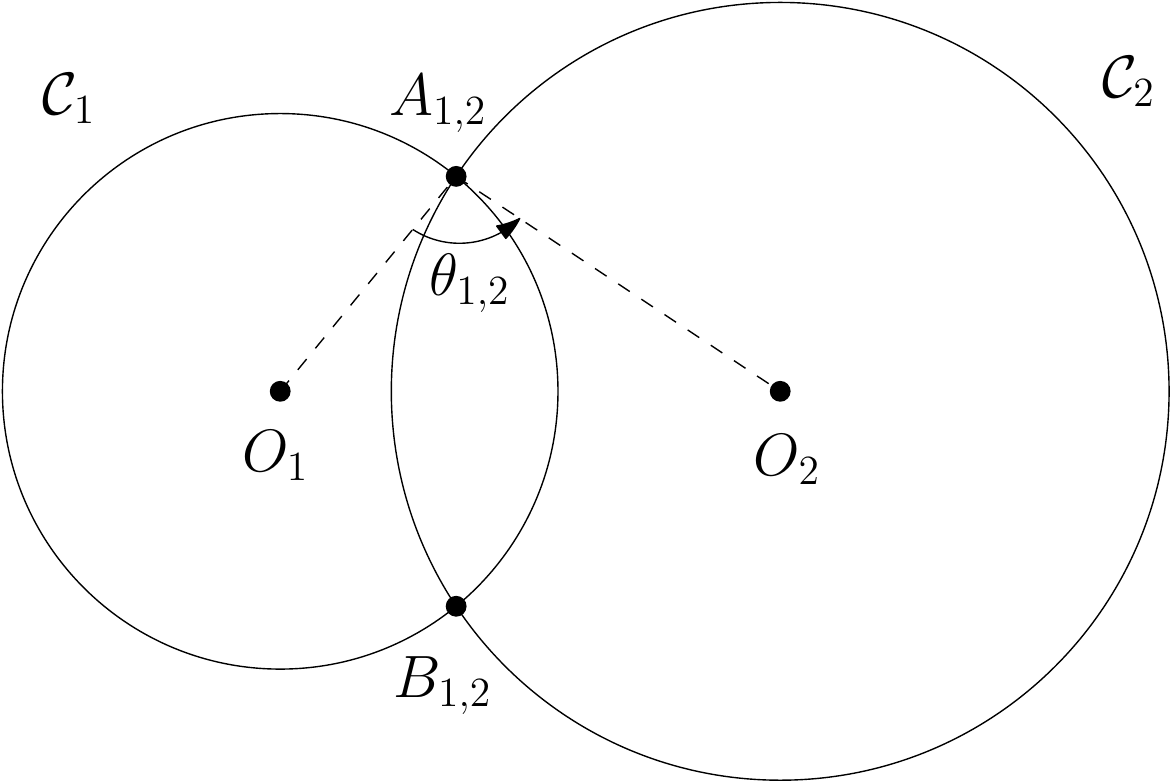}
\caption{Exterior intersection angle $\theta_{1,2}$.}
\label{fig:thetadef}
\end{figure}

We can now state :

\begin{proposition}[Effective Miquel's theorem]
\label{prop:effective}
The following equivalences hold:
\begin{align*}
\theta_{1,2}+\theta_{3,4}=\theta_{2,3}+\theta_{4,1} &\Leftrightarrow A_{1,2},A_{2,3},A_{3,4},A_{4,1} \text{ concyclic } \\
&\Leftrightarrow B_{1,2},B_{2,3},B_{3,4},B_{4,1} \text{ concyclic.}
\end{align*}
\end{proposition}

\begin{proof}
We depict on Figure~\ref{fig:effectiveMiquel} the rectilinear quadrilateral $A_{1,2}A_{2,3}A_{3,4}A_{4,1}$ in dotted lines and the curvilinear quadrilateral $A_{1,2}A_{2,3}A_{3,4}A_{4,1}$ in solid lines.

\begin{figure}[htpb]
\centering
\includegraphics[height=1.5in]{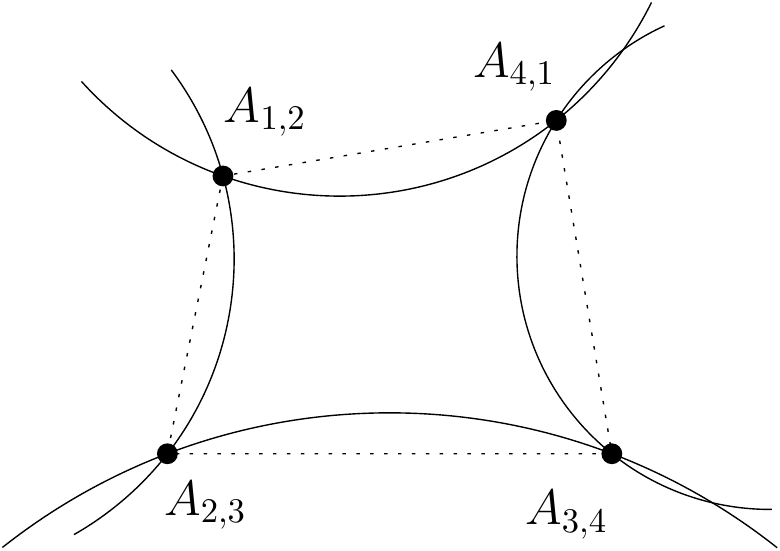}
\caption{The rectilinear quadrilateral and the curvilinear quadrilateral $A_{1,2}A_{2,3}A_{3,4}A_{4,1}$.}
\label{fig:effectiveMiquel}
\end{figure}

Since a chord of a circle makes equal angles with the circle at both endpoints, it follows that the alternating sum $\theta_{1,2}-\theta_{2,3}+\theta_{3,4}-\theta_{4,1}$ of the angles of the curvilinear quadrilateral equals the alternating sum of the angles of the rectilinear quadrilateral. It is well-known that the latter sum is zero if and only if the points $A_{1,2},A_{2,3},A_{3,4}$ and $A_{4,1}$ are concyclic, which concludes the proof.
\end{proof}

Proposition~\ref{prop:effective} provides an effective version of Miquel's theorem. To the best of our knowledge, this characterization has not appeared in the literature yet.

We use Miquel's theorem to define a dynamical system on the space of square grid circle patterns, which we now define and discuss.

\section{Circle patterns and Miquel dynamics}
\label{sec:patterns}

\begin{definition}
A \emph{square grid circle pattern} (SGCP) is a map $S:\Z^2\rightarrow\R^2$ such that the following two conditions are verified :
\begin{enumerate}
\item for any $(x,y)\in\Z^2$, the points $S(x,y),S(x+1,y),S(x+1,y+1)$ and $S(x,y+1)$ are pairwise distinct and concyclic, with their circumcenter denoted by $O_{x+\tfrac{1}{2},y+\tfrac{1}{2}}^S$ ;
\item for any $(i,j)\in\left(\Z+\frac{1}{2}\right)^2$, the circumcenter $O_{i,j}^S$ is distinct from its neighboring circumcenters $O_{i+1,j}^S$ and $O_{i,j+1}^S$.
\end{enumerate}
\end{definition}

In other words, it is a drawing of the square grid such that each face admits a circumcircle. The condition of the four points being pairwise distinct implies in particular that edges cannot collapse to a single point. The second condition implies that dual edges, connecting two neighboring circumcenters, do not collapse to a point. It also implies that for any $(i,j)\in\left(\Z+\frac{1}{2}\right)^2$, the circumcenter $O_{i,j}^S$ is distinct from its diagonally neighboring circumcenters $O_{i+1,j+1}^S$ and $O_{i+1,j-1}^S$. Indeed, if two diagonally opposite circumcircles containing a given point $S(x,y)$ were equal, then all four circumcircles containing that point would be equal. Folds (see Figure~\ref{fig:fold}) and non-convex quadrilaterals (see Figure~\ref{fig:nonconvex}) are allowed.

\begin{figure}[htbp]
\centering
\subfloat[Fold along the edge connecting $(0,1)$ and $(1,1)$.]{\label{fig:fold}\includegraphics[height=2in]{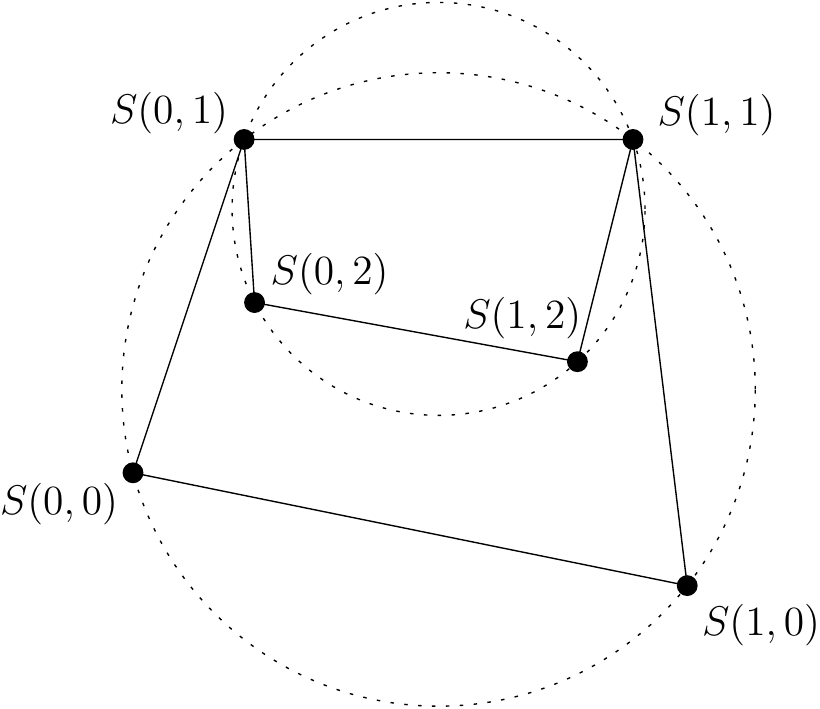}}
\hspace{\stretch{1}}
\subfloat[Face $F_{\tfrac{1}{2},\tfrac{1}{2}}$ is mapped by $S$ to a non-convex quadrilateral.]{\label{fig:nonconvex}\includegraphics[height=1.5in]{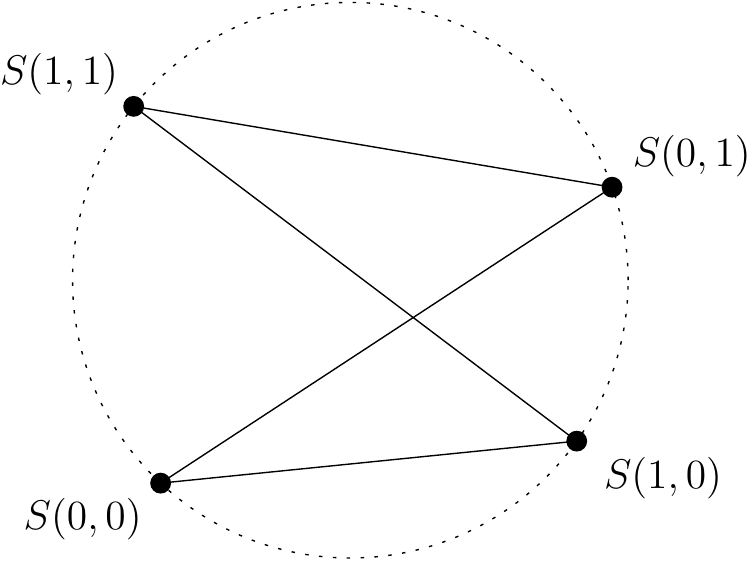}}
\caption{Two allowed local configurations.}
\label{fig:allowed}
\end{figure}

We write $\calS$ for the set of all square grid circle patterns. For any $(i,j)\in\left(\Z+\tfrac{1}{2}\right)^2$, we denote by $F_{i,j}$ the face of $\Z^2$ with vertices $\left(i\pm\tfrac{1}{2},j\pm\tfrac{1}{2}\right)$. Each face of $\Z^2$ is colored black or white in a checkerboard fashion : face $F_{i,j}$ is colored black (resp. white) if $j-i$ is even (resp. odd). If $S\in\calS$, we write $\calC_{i,j}^S$ for the circumcircle of the face $S\left(F_{i,j}\right)$. The circle $\calC_{i,j}^S$ is called black (resp. white) if the face $F_{i,j}$ is black (resp. white). See Figure~\ref{fig:SGCP} for an example.

\begin{figure}[htbp]
\centering
\subfloat{\includegraphics[height=3in]{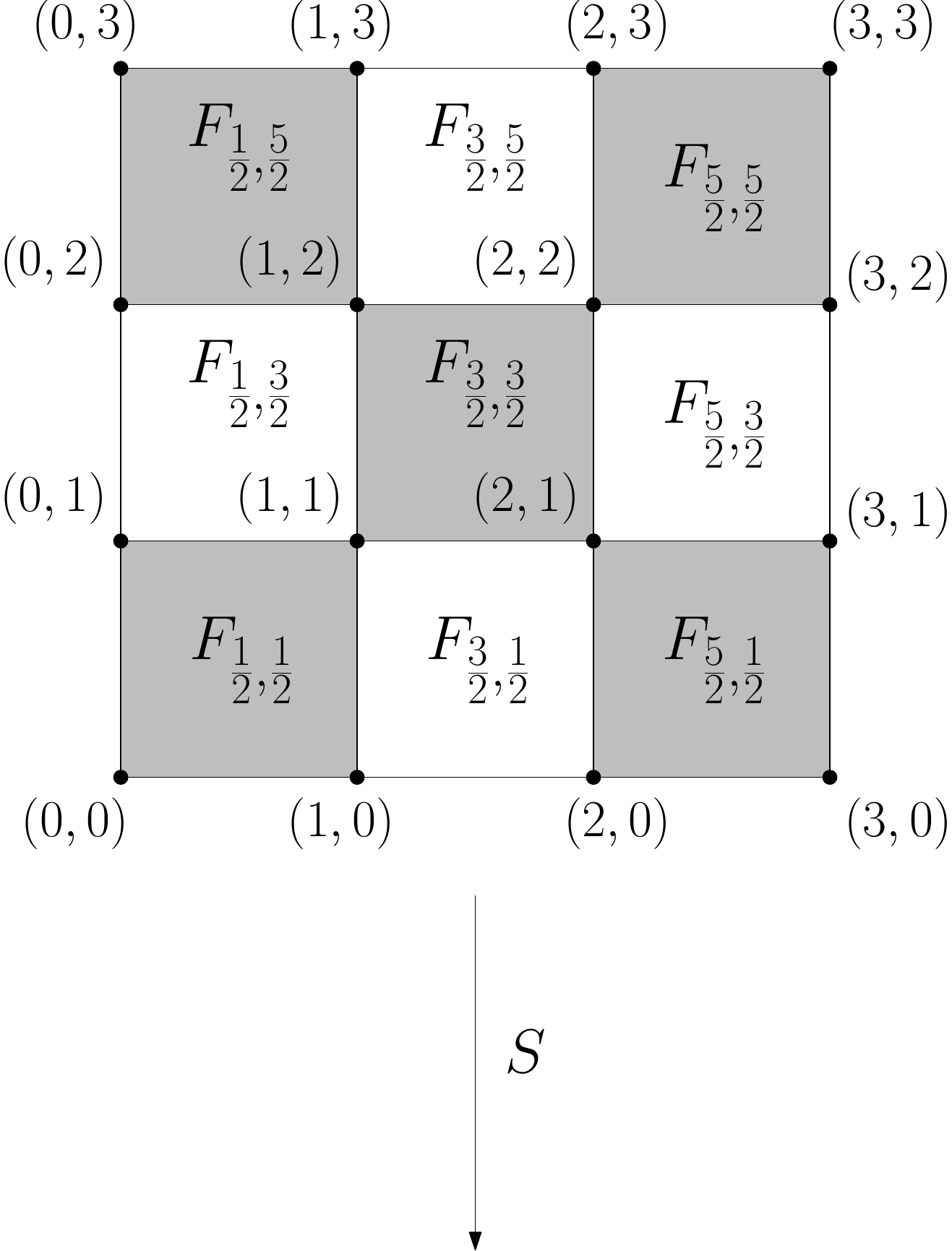}}

\subfloat{\includegraphics[height=1.5in]{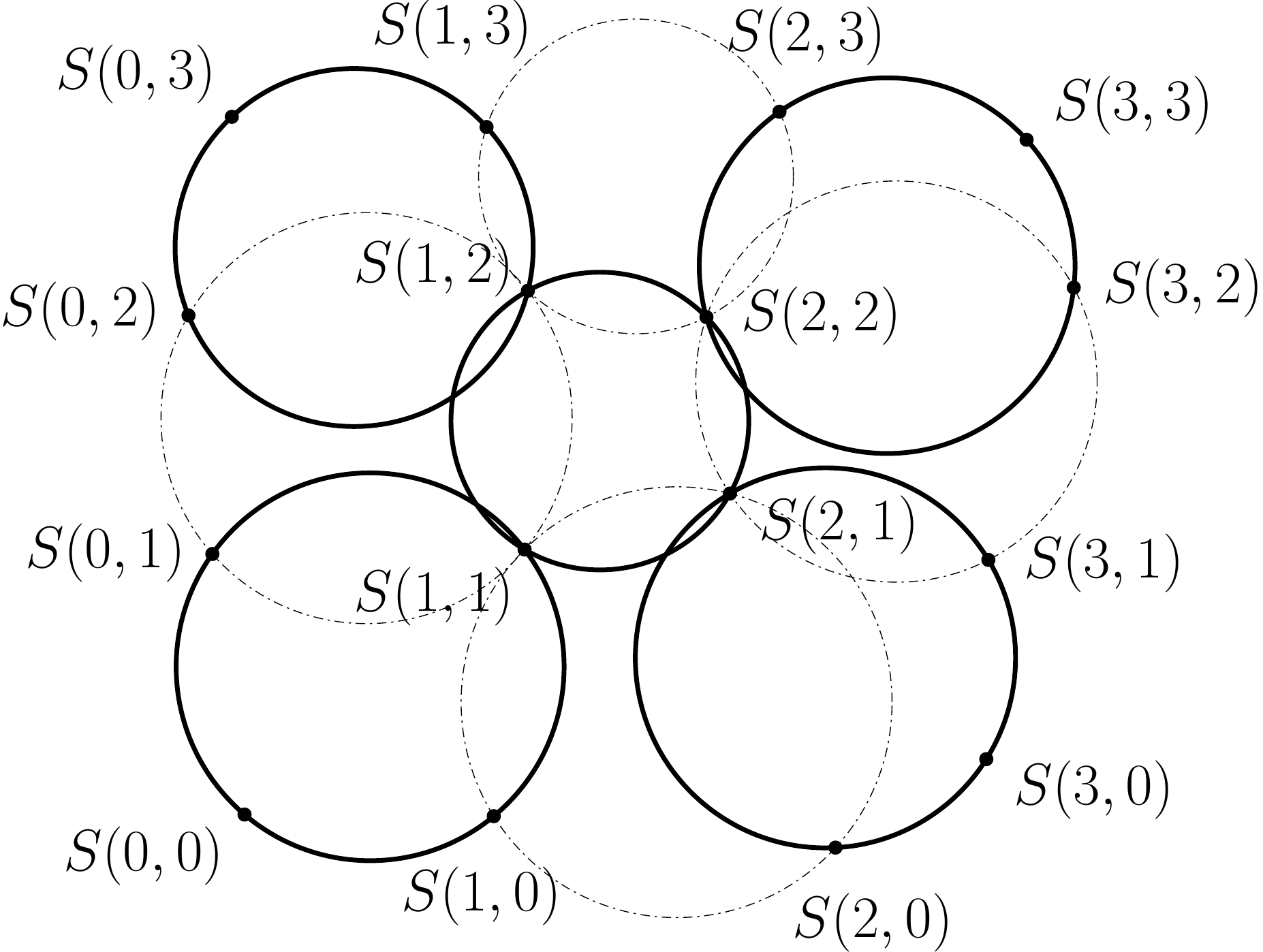}}
\caption{A portion of a square grid circle pattern. Black (resp. white) circles are represented with a bold (resp. dash-dotted) stroke.}
\label{fig:SGCP}
\end{figure}

In order to reduce the problem from an infinite-dimensional system to a finite-dimensional one, we are led to consider spatially biperiodic SGCPs. Let $\vec{a}=(a_1,a_2)$ be a vector in $\Z^2$. We say that an SGCP $S$ is \emph{$\vec{a}$-periodic} if
\[
\exists \vec{u}\in\R^2 \ \forall (x,y)\in\Z^2, S(x+a_1,y+a_2)=S(x,y)+\vec{u}.
\]

Such a vector $\vec{u}$ is called the monodromy in the direction $\vec{a}$. Given two non-collinear integer vectors $\vec{a}$ and $\vec{b}$, we denote by $\calS_{\vec{a},\vec{b}}$ the set of all SGCPs that are both $\vec{a}$-periodic and $\vec{b}$-periodic. It is not hard to see that any $\calS_{\vec{a},\vec{b}}$ is equal to some $\calS_{\vec{a'},\vec{b'}}$, with $\vec{a'}=(m,0)$ and $\vec{b'}=(s,n)$, where $m$ and $n$ are positive integers and $s$ is an integer such that $0\leq s <m$. We will then denote by $\calS_{m,n}^s$ the set $\calS_{(m,0),(s,n)}$. Elements of $\calS_{m,n}^s$ can alternatively be seen as circle patterns drawn on a flat torus. The fundamental domain is an $m$ by $n$ square grid and fundamental domains are glued together with a horizontal shift of $s$ to tile the plane (see Figure~\ref{fig:fundamentaldomains}). In order for the spatial periodicity to be compatible with the checkerboard coloring of the plane (we want a well-defined checkerboard coloring of the square grid on the torus), we will restrict ourselves to the cases when $m$ is even and $s$ has the same parity as $n$.

\begin{figure}[htpb]
\centering
\includegraphics[height=2in]{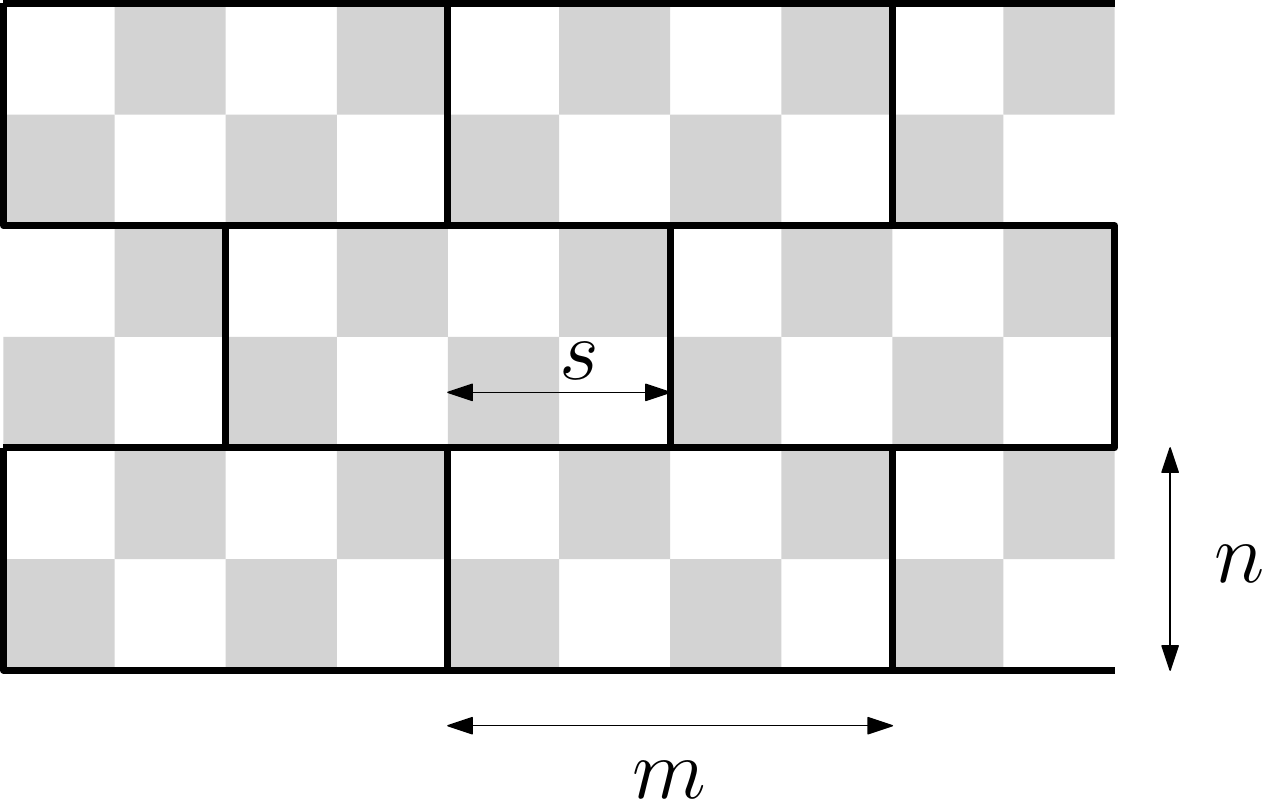}
\caption{A tiling of the plane by fundamental domains (in bold) in the case $(m,n,s)=(4,2,2)$.}
\label{fig:fundamentaldomains}
\end{figure}

We now describe two maps $\mu_B$ and $\mu_W$ from $\calS$ to itself, the definition of which relies on Miquel's theorem. Given $S\in\calS$ we construct $S'=\mu_B(S)$ as follows. If $(x,y)\in\Z^2$ such that $x-y$ is even (resp. odd), we set $S'(x,y)$ to be the symmetric of $S(x,y)$ with respect to the line going through $O_{x-\tfrac{1}{2},y+\tfrac{1}{2}}^S$ and $O_{x+\tfrac{1}{2},y-\tfrac{1}{2}}^S$ (resp. $O_{x+\tfrac{1}{2},y+\tfrac{1}{2}}^S$ and $O_{x-\tfrac{1}{2},y-\tfrac{1}{2}}^S$). In other words, each vertex of $S$ gets moved to the other intersection point of the two white circles it belongs to.

We now check that, for generic $S$, $S'$ is an SGCP. Fix $(i,j)\in\left(\Z+\tfrac{1}{2}\right)^2$. If $i-j$ is odd, then the points $S'\left(i\pm\tfrac{1}{2},j\pm\tfrac{1}{2}\right)$ are on the circle $\calC_{i,j}^S$, thus concyclic (white circles don't change under $\mu_B$). If $i-j$ is even, we apply Miquel's theorem, with the circles $\calC_{i\pm1,j}^S$ and $\calC_{i,j\pm1}^S$ playing the roles of $\calC_k$, and the points $S\left(i\pm\tfrac{1}{2},j\pm\tfrac{1}{2}\right)$ (resp. $S'\left(i\pm\tfrac{1}{2},j\pm\tfrac{1}{2}\right)$) playing the roles of the $A_{k,k+1}$ (resp. $B_{k,k+1}$). The points $S\left(i\pm\tfrac{1}{2},j\pm\tfrac{1}{2}\right)$ are concyclic because they lie on the circle $\calC_{i,j}^S$, thus the points $S'\left(i\pm\tfrac{1}{2},j\pm\tfrac{1}{2}\right)$ are concyclic\footnote{Strictly speaking, the points $S'\left(i\pm\tfrac{1}{2},j\pm\tfrac{1}{2}\right)$ are concyclic on the Riemann sphere and could be aligned in the plane. This however does not occur for generic $S$ and can even be ruled out in every case by applying a M\"obius transformation.} and lie on some circle $\calC_{i,j}^{S'}$. For generic $S$, points around a face of $S'$ and neighboring circumcenters of $S'$ are distinct, so $S'$ is an SGCP with the same white circles as $S$ (only the black circles may have changed).

We similarly define $\mu_W(S)$ by moving each vertex $S(i,j)$ to the other intersection point of the two black circles it belongs to (when these two circles are tangent, the vertex does not move).

The maps $\mu_B$ and $\mu_W$ are respectively called \emph{black mutation} and \emph{white mutation}. Clearly, they are both involutions.

Miquel dynamics is defined by composing $\mu_B$ and $\mu_W$ in an alternating fashion.  Fix $S_0$ to be the initial SGCP. We define the bi-infinite sequence $\left(S_t\right)_{t\in\Z}$ by :
\begin{itemize}
\item $S_{2t}=\left(\mu_W\circ\mu_B\right)^t(S_0)$ and $S_{2t+1}=\mu_B\circ\left(\mu_W\circ\mu_B\right)^t(S_0)$ if $t$ is a positive integer ;
\item  $S_{2t}=\left(\mu_B\circ\mu_W\right)^{-t}(S_0)$ and $S_{2t+1}=\mu_W\circ\left(\mu_B\circ\mu_W\right)^{-t}(S_0)$ if $t$ is a negative integer.
\end{itemize}

Here, $\left(\mu_W\circ\mu_B\right)^t$ denotes the composition of the map $\mu_W\circ\mu_B$ with itself $t$ times.

If an SGCP $S$ is $\vec{a}$-periodic for some $\vec{a}\in\Z^2$ then $\mu_B(S)$ and $\mu_W(S)$ are also $\vec{a}$-periodic, thus Miquel dynamics is also a well-defined dynamics on $\calS_{m,n}^s$, the set of $m$ by $n$ circle patterns on a torus with a shift of $s$. This is now a dynamical system on a finite-dimensional space.

We say that two SGCPs $S$ and $S'$ are equivalent and denote it by $S\sim S'$ if there exists a similarity $\sigma$ (composition of a rotation, a homothety and a translation) such that $S'=\sigma\circ S$. We define the space of circle patterns up to similarity $\calM_{m,n}^s:=\calS_{m,n}^s/\sim$. The maps $\mu_B$ and $\mu_W$ commute with similarities, thus Miquel dynamics descends to the space $\calM_{m,n}^s$.

\section{Coordinates for biperiodic circle patterns}
\label{sec:coordinates}

In this section we define some coordinates on the space $\calM_{m,n}^s$, which will be some angles $\phi$ satisfying certain relations.

Define $\Gamma=\left(\Z+\frac{1}{2}\right)^2$ and $\calD=\left\{N,W,S,E\right\}$. If $\phi$ is a map from $\Gamma\times\calD$ to $\left(0,2\pi\right)$ and if $\left(i,j,D\right)\in\Gamma\times\calD$, we simply denote $\phi\left(i,j,D\right)$ by $\phi_{i,j}^D$. Each face of $\Z^2$ is divided into four triangular regions (north, west, south and east) and elements in $\Gamma\times\calD$ are in bijection with these triangular regions (see Figure~\ref{fig:phiregions}).

\begin{figure}[htpb]
\centering
\includegraphics[height=2in]{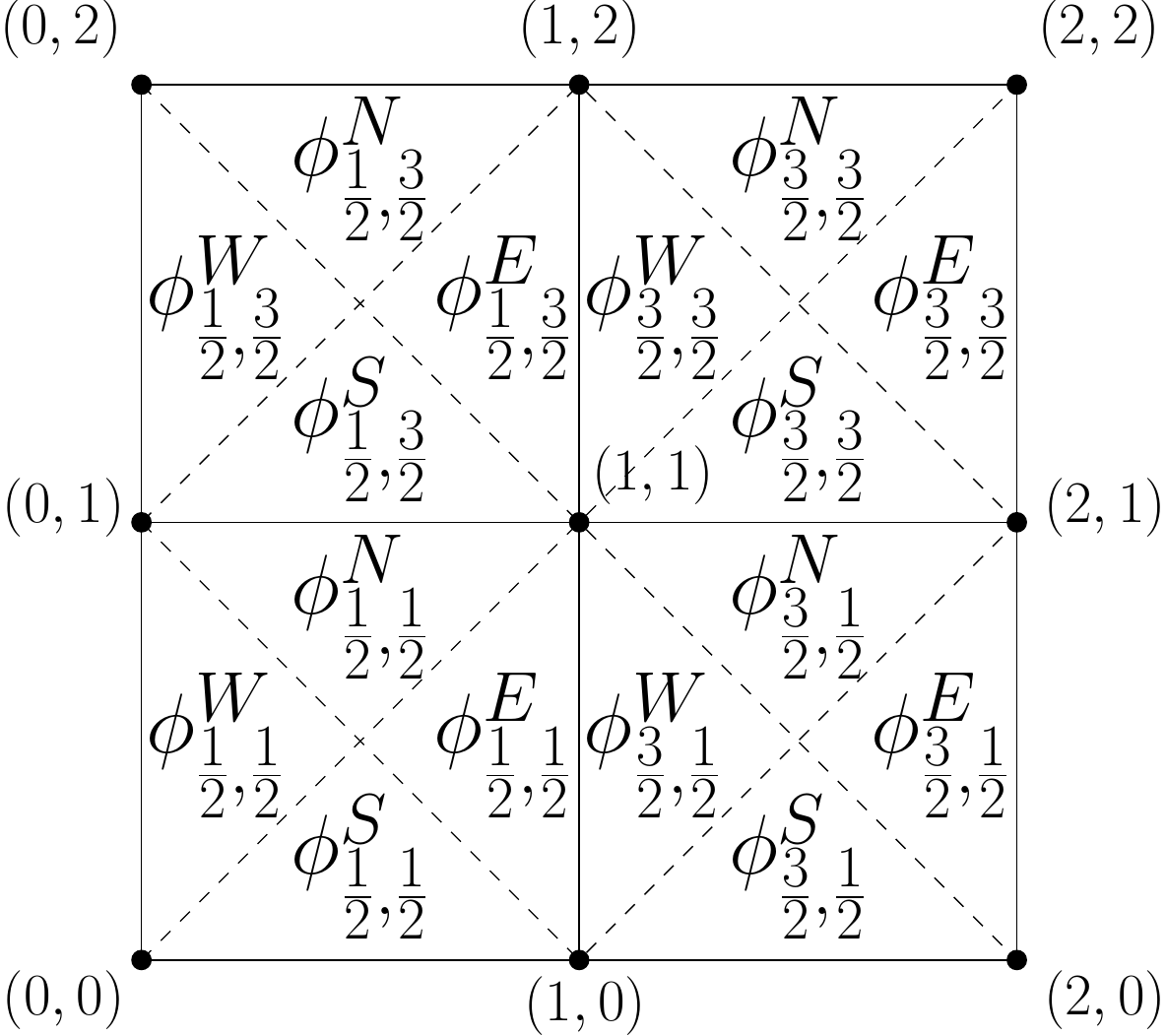}
\caption{Labeling of the triangular regions of a fundamental domain by $\phi$ variables in the case when $m=n=2$.}
\label{fig:phiregions}
\end{figure}

Given a circle pattern $S\in\calS_{m,n}^s$, we define for any $(i,j,D)\in\Gamma\times\calD$ the angle $\phi_{i,j}^D\in\left(0,2\pi\right)$ by the following formulas :
\begin{align}
\phi_{i,j}^N&= \angle S\left(i+\frac{1}{2},j+\frac{1}{2}\right)O_{i,j}S\left(i-\frac{1}{2},j+\frac{1}{2}\right) \\
\phi_{i,j}^W&= \angle S\left(i-\frac{1}{2},j+\frac{1}{2}\right)O_{i,j}S\left(i-\frac{1}{2},j-\frac{1}{2}\right) \\
\phi_{i,j}^S&= \angle S\left(i-\frac{1}{2},j-\frac{1}{2}\right)O_{i,j}S\left(i+\frac{1}{2},j-\frac{1}{2}\right) \\
\phi_{i,j}^E&= \angle S\left(i+\frac{1}{2},j-\frac{1}{2}\right)O_{i,j}S\left(i+\frac{1}{2},j+\frac{1}{2}\right).
\end{align}
They are the angles under which we see an edge of a face (the northern, western, southern or eastern boundary of the face) when we are standing at the circumcenter of the face (see Figure~\ref{fig:phidefinition}).

\begin{figure}[htpb]
\centering
\includegraphics[height=2in]{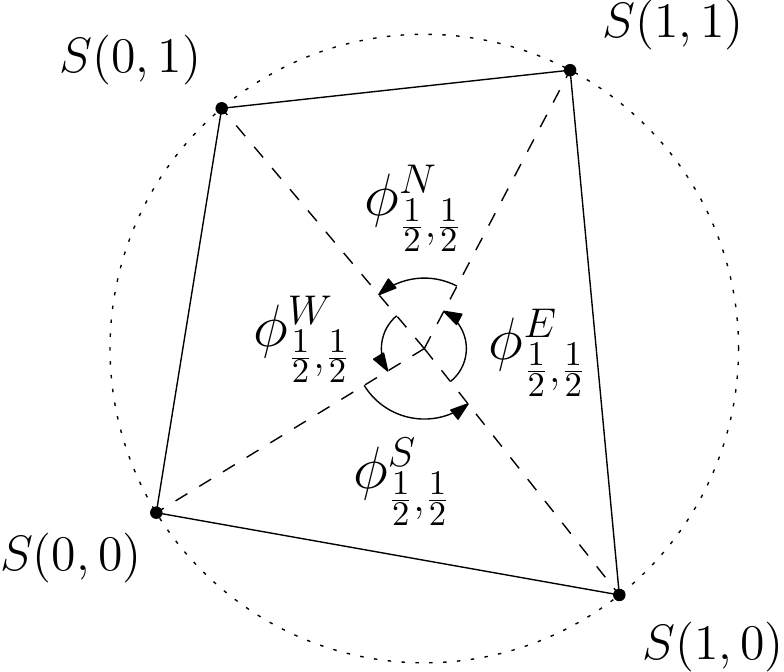}
\caption{The $\phi$ angles are the angles under which we see an edge while standing at the circumcenter of a face.}
\label{fig:phidefinition}
\end{figure}

These angles are left unchanged if a similarity is applied to $S$, thus this provides a well-defined map $\Phi_{m,n}^s$ from $\calM_{m,n}^s$ to $\left(0,2\pi\right)^{\Gamma\times\calD}$. In order to use the $\phi$ angles to coordinatize the space $\calM_{m,n}^s$, we need to show that $\Phi_{m,n}^s$ realizes a bijection from $\calM_{m,n}^s$ onto its image, which will be a certain subset $\calA_{m,n}^s$ of $\left(0,2\pi\right)^{\Gamma\times\calD}$. Before we state the equations cutting out that subset, we need some additional notation.

For any $(x,y)\in\Z^2$, we define the following sets of triangles :
\begin{align}
T_b(x,y)&:=\bigg\{\left(x+\frac{1}{2},y+\frac{1}{2},W\right),\left(x-\frac{1}{2},y+\frac{1}{2},S\right), \nonumber \\
&\ \ \ \ \ \ \ \left(x-\frac{1}{2},y-\frac{1}{2},E\right),\left(x+\frac{1}{2},y-\frac{1}{2},N\right)\bigg\};\\
T_w(x,y)&:=\bigg\{\left(x+\frac{1}{2},y+\frac{1}{2},S\right),\left(x-\frac{1}{2},y+\frac{1}{2},E\right), \nonumber \\
&\ \ \ \ \ \ \ \left(x-\frac{1}{2},y-\frac{1}{2},N\right),\left(x+\frac{1}{2},y-\frac{1}{2},W\right)\bigg\};\\
T(x,y)&:=T_b(x,y) \cup T_w(x,y).
\end{align}
The set $T(x,y)$ denotes the eight triangles surrounding vertex $(x,y)$, while $T_b(x,y)$ (resp. $T_w(x,y)$) denotes the four black (resp. white) triangles surrounding that vertex, where the triangles are colored in an alternating fashion as in Figure~\ref{fig:vertexflower}.

\begin{figure}[htpb]
\centering
\includegraphics[height=2in]{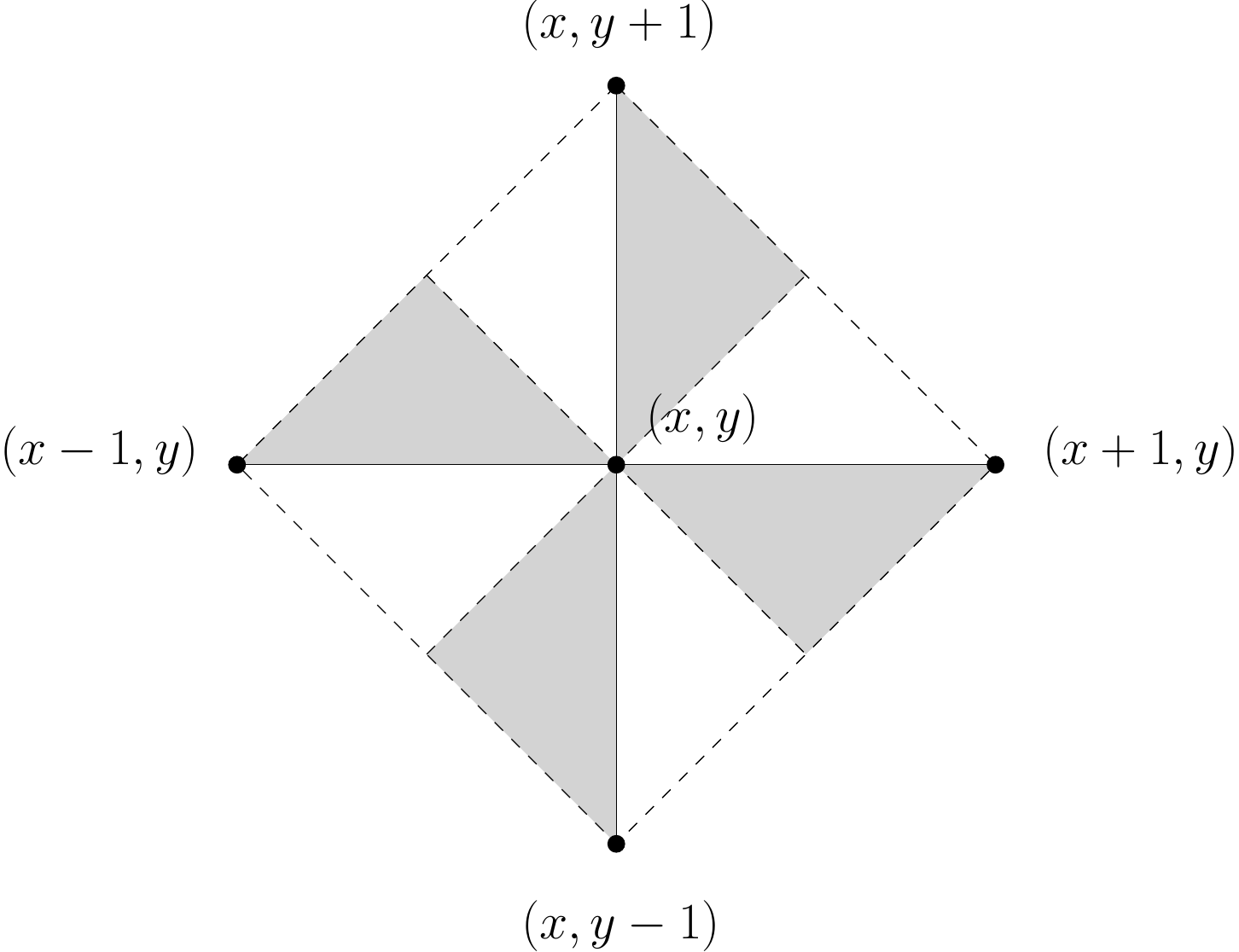}
\caption{The black and white triangles around vertex $(x,y)$.}
\label{fig:vertexflower}
\end{figure}

We now define what will be the image of the map $\Phi_{m,n}^s$ and for simplicity, we begin with the case when the shift $s=0$ (which implies that $n$ is even).

\begin{definition}
We define $\calA_{m,n}^0$ to be the set of all $\phi\in\left(0,2\pi\right)^{\Gamma\times\calD}$ verifying the following relations :
\begin{gather}
\forall \left(i,j\right)\in\Gamma, \forall D\in\calD,\ \phi_{i,j}^D=\phi_{i+m,j}^D=\phi_{i,j+n}^D;\label{eq:biperiodicity}\\
\forall \left(i,j\right)\in\Gamma, \ \phi_{i,j}^N+\phi_{i,j}^W \not\equiv0 \mod 2\pi;\label{eq:diag1}\\
\forall \left(i,j\right)\in\Gamma, \ \phi_{i,j}^N+\phi_{i,j}^E\not\equiv0 \mod 2\pi;\label{eq:diag2}\\
\forall \left(i,j\right)\in\Gamma, \ \sum_{D\in\calD}\phi_{i,j}^D \equiv0 \mod 2\pi; \label{eq:faceflat}\\
\forall \left(i,j\right)\in\Gamma, \ \sum_{\tau\in T\left(i+\frac{1}{2},j+\frac{1}{2}\right)}\phi(\tau) \equiv0 \mod 4\pi; \label{eq:vertexflat}\\
\forall \left(i,j\right)\in\Gamma, \ \prod_{\tau\in T_b\left(i+\frac{1}{2},j+\frac{1}{2}\right)}\sin\frac{\phi(\tau)}{2}=\prod_{\tau\in T_w\left(i+\frac{1}{2},j+\frac{1}{2}\right)}\sin\frac{\phi(\tau)}{2}; \label{eq:vertexmonodromy} \\
\forall i\in\left(\Z+\frac{1}{2}\right)\cap\left(0,m\right), \ \prod_{\substack{j\in\Z+\frac{1}{2}\\0<j<n}}\sin\frac{\phi_{i,j}^N}{2}=\prod_{\substack{j\in\Z+\frac{1}{2}\\0<j<n}}\sin\frac{\phi_{i,j}^S}{2};\label{eq:verticalmonodromy}\\
\forall j\in\left(\Z+\frac{1}{2}\right)\cap\left(0,n\right), \ \prod_{\substack{i\in\Z+\frac{1}{2}\\0<i<m}}\sin\frac{\phi_{i,j}^W}{2}=\prod_{\substack{i\in\Z+\frac{1}{2}\\0<i<m}}\sin\frac{\phi_{i,j}^E}{2}; \label{eq:horizontalmonodromy} \\
\forall i\in\left(\Z+\frac{1}{2}\right)\cap\left(0,m\right), \ \sum_{\substack{j\in\Z+\frac{1}{2}\\0<j<n}}\phi_{i,j}^W\equiv\sum_{\substack{j\in\Z+\frac{1}{2}\\0<j<n}}\phi_{i,j}^E \mod 2\pi; \label{eq:verticalparallel}\\
\forall j\in\left(\Z+\frac{1}{2}\right)\cap\left(0,n\right), \ \sum_{\substack{i\in\Z+\frac{1}{2}\\0<i<m}}\phi_{i,j}^N\equiv\sum_{\substack{i\in\Z+\frac{1}{2}\\0<i<m}}\phi_{i,j}^S \mod 2\pi. \label{eq:horizontalparallel}
\end{gather}
\end{definition}
Because of the biperiodicity assumption~\eqref{eq:biperiodicity}, there are only finitely many degrees of freedom in $\calA_{m,n}^0$.
We can now state the following coordinatization result for $\calM_{m,n}^0$ :
\begin{theorem}
\label{thm:coordinatization}
$\Phi_{m,n}^0$ is a bijection from $\calM_{m,n}^0$ to $\calA_{m,n}^0$.
\end{theorem}

\begin{proof}
Let us first show that conditions~\eqref{eq:biperiodicity} to~\eqref{eq:horizontalparallel} are necessary. The first condition~\eqref{eq:biperiodicity} follows from the spatial biperiodicity of circle patterns in $\calS_{m,n}^0$. The next five conditions are local conditions satisfied by any circle pattern in the plane, while the last four are a consequence of the spatial biperiodicity. Let $S$ be a circle pattern in $\calS_{m,n}^0$, denote by $\left[S\right]$ its class in $\calM_{m,n}^0$ and write $\phi=\Phi_{m,n}^0(\left[S\right])$. Fix $(i,j)\in\Gamma$. The relations~\eqref{eq:diag1} and~\eqref{eq:diag2} are a consequence of the vertices of $S\left(F_{i,j}\right)$ being pairwise distinct. Equation~\eqref{eq:faceflat} expresses flatness at $O_{i,j}^S$, the circumcenter of the face $S\left(F_{i,j}\right)$. Equation~\eqref{eq:vertexflat} expresses flatness at the vertex $S\left(i+\tfrac{1}{2},j+\tfrac{1}{2}\right)$ and uses the fact that the eight triangles around the vertex are isosceles. Equation~\eqref{eq:vertexmonodromy} accounts for the absence of monodromy around the vertex $S\left(i+\tfrac{1}{2},j+\tfrac{1}{2}\right)$. Indeed, if we denote by $R_{i,j}^S$ the radius of the circle $\calC_{i,j}^S$, then we have
\begin{align}
\frac{R_{i+1,j}^S}{R_{i,j}^S} &= \frac{\sin \frac{\phi_{i,j}^E}{2}}{\sin \frac{\phi_{i+1,j}^W}{2}} \label{eq:verticaledgeradius} \\
\frac{R_{i,j+1}^S}{R_{i,j}^S} &= \frac{\sin \frac{\phi_{i,j}^N}{2}}{\sin \frac{\phi_{i,j+1}^S}{2}}. \label{eq:horizontaledgeradius}
\end{align}
Note that because we chose the $\phi$ angles to take values in $\left(0,2\pi\right)$, the quantities $\sin\tfrac{\phi}{2}$ are always positive. Equation~\eqref{eq:vertexmonodromy} follows from rewriting
\[
\frac{R_{i,j}^S}{R_{i+1,j}^S}\frac{R_{i+1,j}^S}{R_{i+1,j+1}^S}\frac{R_{i+1,j+1}^S}{R_{i,j+1}^S}\frac{R_{i,j+1}^S}{R_{i,j}^S}=1.
\]
Equation~\eqref{eq:verticalmonodromy} expresses the absence of monodromy when going across the torus vertically (translations don't change the radii of the circles), and is obtained  by substituting equations of the type of~\eqref{eq:horizontaledgeradius} into
\[
\prod_{\substack{j\in\Z+\frac{1}{2}\\0<j<n}}\frac{R_{i,j+1}^S}{R_{i,j}^S}=1.
\]
Equation~\eqref{eq:horizontalmonodromy} follows similarly.
Equation~\eqref{eq:verticalparallel} is a consequence of the invariance by translation. Fix $i\in\left(\Z+\tfrac{1}{2}\right)$ and write $x=i-\tfrac{1}{2}$. By translation invariance, the angle between the vectors $\overrightarrow{S(x+1,0)S(x,0)}$ and $\overrightarrow{S(x,n)S(x+1,n)}$ equals $\pi$. Walking along the edges $\overrightarrow{S(x,y)S(x,y+1)}$ for $0 \leq y \leq n-1$ and counting the turning angles (see Figure~\ref{fig:turningangles}), we obtain that
\begin{multline}
\sum_{0\leq y \leq n-1} \pi - \angle S(x+1,y)S(x,y)S(x,y+1)- \angle S(x,y)S(x,y+1)S(x+1,y+1) \\
\equiv 0 \mod 2\pi.
\end{multline}
The conclusion follows from the observation that, using~\eqref{eq:faceflat}, one can write
\begin{multline}
\pi - \angle S(x+1,y)S(x,y)S(x,y+1)- \angle S(x,y)S(x,y+1)S(x+1,y+1) \\
=\frac{1}{2}\left(\phi_{x+\frac{1}{2},y+\frac{1}{2}}^W-\phi_{x+\frac{1}{2},y+\frac{1}{2}}^E\right).
\end{multline}
Equation~\eqref{eq:horizontalparallel} is proved similarly.

\begin{figure}[htpb]
\centering
\includegraphics[height=2in]{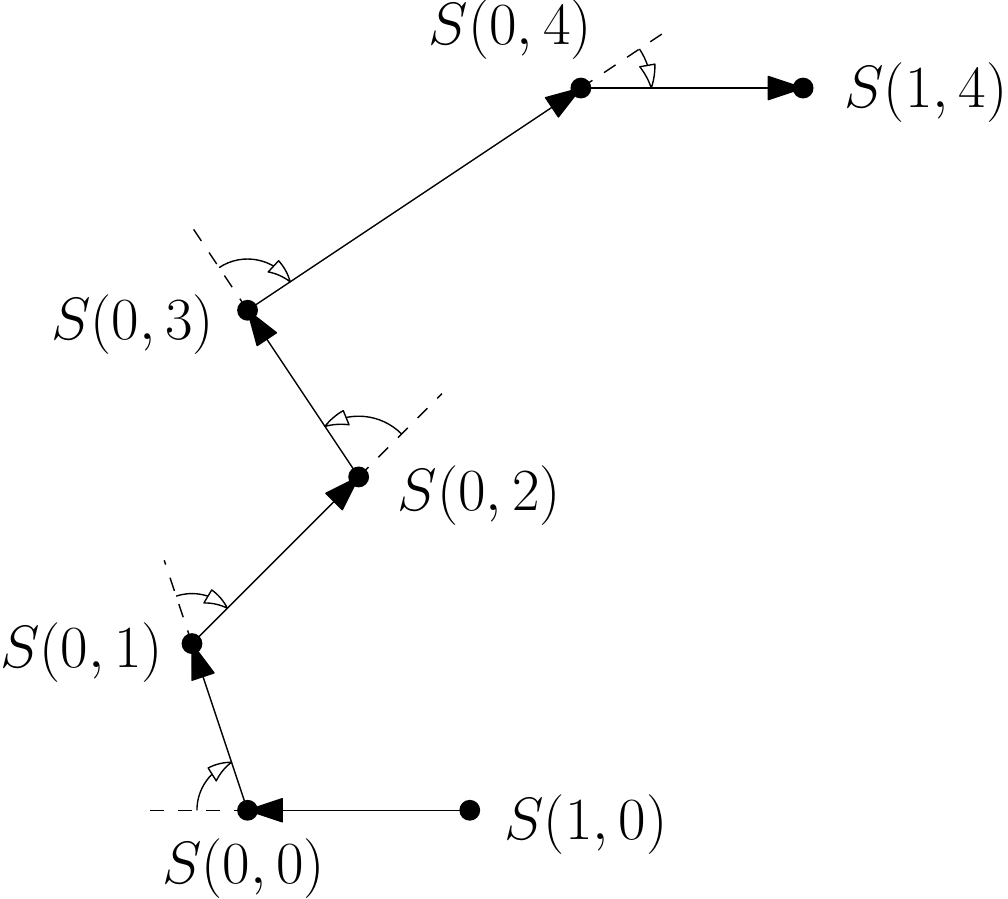}
\caption{Count of the turning angles for a walk crossing the torus vertically, in the case when $x=0$ and $n=4$.}
\label{fig:turningangles}
\end{figure}

Reciprocally, fix $\phi\in \calA_{m,n}^0$. We will construct the unique pattern $\hat{S}\in \calS_{m,n}^0$ such that $\Phi_{m,n}^0(\hat{S})=\phi$, $\hat{S}(0,0)=(0,0)$ and $\hat{S}(1,0)=(1,0)$. The class $[\hat{S}]$ will be the unique preimage of $\phi$ in $\calM_{m,n}^0$. We start with $\hat{S}(0,0)=(0,0)$ and $\hat{S}(1,0)=(1,0)$. Knowing $\phi_{\tfrac{1}{2},\tfrac{1}{2}}^S\in\left(0,2\pi\right)$ determines $O_{\tfrac{1}{2},\tfrac{1}{2}}^{\hat{S}}$ and the circle $\calC_{\tfrac{1}{2},\tfrac{1}{2}}^{\hat{S}}$. From there, $\phi_{\tfrac{1}{2},\tfrac{1}{2}}^W$ and $\phi_{\tfrac{1}{2},\tfrac{1}{2}}^E$ enable us to reconstruct $\hat{S}(0,1)$ and $\hat{S}(1,1)$. We repeat this procedure : knowing the position of the two endpoints of an edge, we can reconstruct the other two points of a face containing that edge as well as its circumcircle. Thus, we reconstruct iteratively all the vertices and centers. Equations~\eqref{eq:vertexflat} and~\eqref{eq:vertexmonodromy} around a vertex guarantee that the construction of the center $(i,j)$ does not depend on the path chosen from $(\tfrac{1}{2},\tfrac{1}{2})$ to $(i,j)$. Equations~\eqref{eq:diag1}, \eqref{eq:diag2} and \eqref{eq:faceflat} ensure that all four points around a given face are pairwise distinct. Equations~\eqref{eq:faceflat} guarantee that $\Phi_{m,n}^0(\hat{S})=\phi$. To prove that the top boundary of a fundamental domain is obtained as a translate of its bottom boundary, it suffices to show that $\overrightarrow{\hat{S}(x,0)\hat{S}(x+1,0)}=\overrightarrow{\hat{S}(x,n)\hat{S}(x+1,n)}$ for any $0 \leq x\leq m-1$. These two vectors have the same direction by~\eqref{eq:verticalparallel} and the same norm by~\eqref{eq:verticalmonodromy}. The fact that the right boundary is a translate of the left boundary follows similarly from~\eqref{eq:horizontalparallel} and~\eqref{eq:horizontalmonodromy}. The spatial biperiodicity of $\hat{S}$ is a consequence of~\eqref{eq:biperiodicity}.
\end{proof}

\begin{remark}
We stress that, while the $\phi$ coordinates described above are an appropriate framework to derive local recurrence formulas as we do in the next section, they are not convenient at all to construct an element of $\calM_{m,n}^0$. Indeed, it is not straightforward to produce a $(4mn)$-tuple of variables satisfying conditions~\eqref{eq:biperiodicity} to~\eqref{eq:horizontalparallel}. It is actually not even clear from these coordinates what the dimension of $\calM_{m,n}^0$ is. Still, there are two observations we can make to help understand the space $\calM_{m,n}^0$ better.

Firstly, it follows from Theorem 3 of~\cite{BS} that, if we fix the intersection angles  $\theta$ of each pair of neighboring circles so that they sum to $2\pi$ at each vertex, then, provided they satisfy some additional inequalities, they determine a unique element of $\calM_{m,n}^0$ (this element lives in an open strict subset of $\calM_{m,n}^0$, since only cell decompositions are considered in~\cite{BS}, which forbids the existence of folds and non-convex quadrilaterals). It is not hard to show that the dimension of such an assignment of $\theta$ variables at each edge with prescribed sum at each vertex is $mn+1$, hence one should expect the space $\calM_{m,n}^0$ to be of dimension $mn+1$. Note however that one cannot reconstruct the circle pattern locally by knowing these $\theta$ angles only in a region of the graph, thus the $\theta$ variables cannot be used to obtain local recurrence formulas.

Secondly, when $n=2$, one can show that, starting from the knowledge of the assignment of $2m+1$ independent $\phi$ variables in the faces indicated by a dot on Figure~\ref{fig:phireconstruction}, one can progressively reconstruct all the other $\phi$ variables (either one at a time, or two at a time)\footnote{We point out that, since some of the reconstruction equation are quadratic, there are actually four possible $8m$-tuples of $\phi$ variables which correspond to the original $2m+1$-tuple of $\phi$ variables.}.

\begin{figure}[htpb]
\centering
\includegraphics[height=1in]{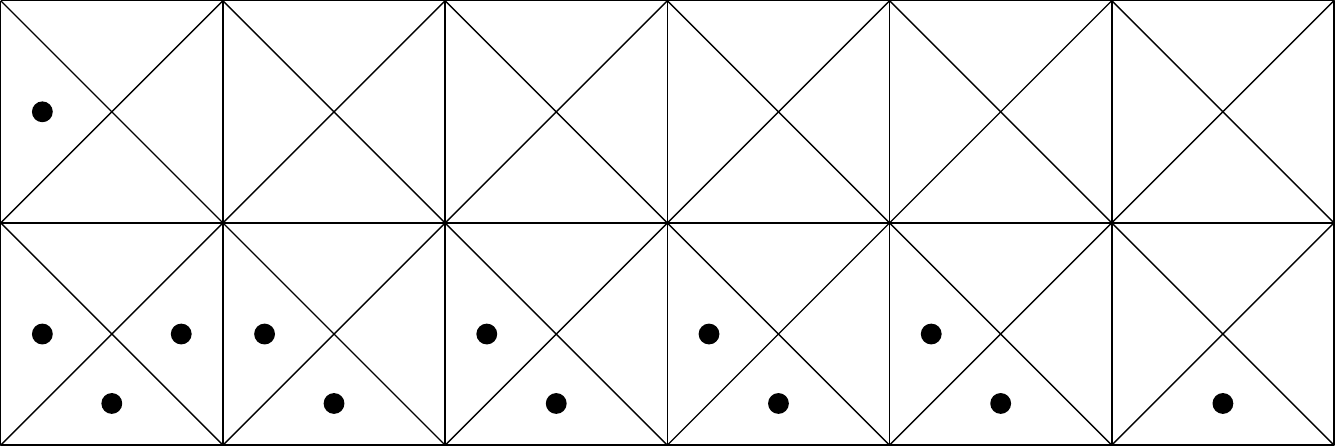}
\caption{A subset of $2m+1$ $\phi$ variables from which one can reconstruct all the other $\phi$ variables in the $m$ by $2$ case.}
\label{fig:phireconstruction}
\end{figure}

This makes the $\phi$ variables convenient to use in the $m$ by $2$ case. It would be interesting to have such a reconstruction result for the general $m$ by $n$ case.
\end{remark}

In the case when $s\neq0$, we define $\calA_{m,n}^s$ in a similar fashion as for $\calA_{m,n}^0$, simply replacing equations~\eqref{eq:biperiodicity},\eqref{eq:verticalmonodromy} and~\eqref{eq:verticalparallel} respectively by
\begin{equation}
\label{eq:biperiodicityshift}
\forall \left(i,j\right)\in\Gamma, \forall D\in\calD,\ \phi_{i,j}^D=\phi_{i+m,j}^D=\phi_{i+s,j+n}^D;
\end{equation}
\begin{multline}
\label{eq:verticalmonodromyshift}
\forall i_0\in\left(\Z+\frac{1}{2}\right)\cap\left(0,m\right), \\
\prod_{\substack{j\in\Z+\frac{1}{2}\\0<j<n}}\sin\frac{\phi_{i_0,j}^S}{2}\prod_{i=i_0+1}^{i_0+s}\sin\frac{\phi_{i,n-\frac{1}{2}}^W}{2}= \prod_{\substack{j\in\Z+\frac{1}{2}\\-1<j<n-1}}\sin\frac{\phi_{i_0,j}^N}{2}\prod_{i=i_0}^{i_0+s-1}\sin\frac{\phi_{i,n-\frac{1}{2}}^E}{2};
\end{multline}
\begin{multline}
\label{eq:verticalparallelshift}
\forall i_0\in\left(\Z+\frac{1}{2}\right)\cap\left(0,m\right), \\
\sum_{\substack{j\in\Z+\frac{1}{2}\\-1<j<n-1}}\phi_{i_0,j}^W + \sum_{i=i_0+1}^{i_0+s}\phi_{i,n-\frac{1}{2}}^N= \sum_{\substack{j\in\Z+\frac{1}{2}\\0<j<n}}\phi_{i_0,j}^E + \sum_{i=i_0}^{i_0+s-1}\phi_{i,n-\frac{1}{2}}^S \mod 2\pi.
\end{multline}
Then the statement of Theorem~\ref{thm:coordinatization} has the following straightforward extension:
\begin{theorem}
\label{thm:coordinatizationshift}
$\Phi_{m,n}^s$ is a bijection from $\calM_{m,n}^s$ to $\calA_{m,n}^s$.
\end{theorem}

\section{Recurrence formulas}
\label{sec:recurrenceformulas}

In the previous section, we associated to an SGCP $S\in\calS_{m,n}^s$ a set of coordinates $\phi_{k,l}^D\in\left(0,2\pi\right)$ with $(k,l,D)\in\Gamma\times\calD$. We define for any $(k,l,D)\in\Gamma\times\calD$ the variables $X_{k,l}^D=e^{i\phi_{k,l}^D}$. In this section, we give formulas describing the action of the maps $\mu_B$ and $\mu_W$ on the variables $X$. Recall that $F_{\tfrac{1}{2},\tfrac{1}{2}}$ is a black face. We will explain how to compute the new values of $X_{\tfrac{1}{2},\tfrac{1}{2}}^N$ (variable associated with a black face) and $X_{\tfrac{1}{2},\tfrac{3}{2}}^S$ (variable associated with a white face) after applying $\mu_B$, the other formulas will follow by symmetry.

Fix $S\in\calS_{m,n}^s$. We will use the following notation for the $X$ coordinates of $S$ in order to alleviate the formulas (see Figure~\ref{fig:XYrenamed}) :
\begin{align*}
&X_N=X_{\frac{1}{2},\frac{1}{2}}^N &&X_W=X_{\frac{1}{2},\frac{1}{2}}^W &&&X_S=X_{\frac{1}{2},\frac{1}{2}}^S &&&&X_E=X_{\frac{1}{2},\frac{1}{2}}^E \\
&Y_N=X_{\frac{1}{2},\frac{3}{2}}^S &&Y_W=X_{-\frac{1}{2},\frac{1}{2}}^E &&&Y_S=X_{\frac{1}{2},-\frac{1}{2}}^N &&&&Y_E=X_{\frac{3}{2},\frac{1}{2}}^W.
\end{align*}
\begin{figure}[htpb]
\centering
\includegraphics[height=2in]{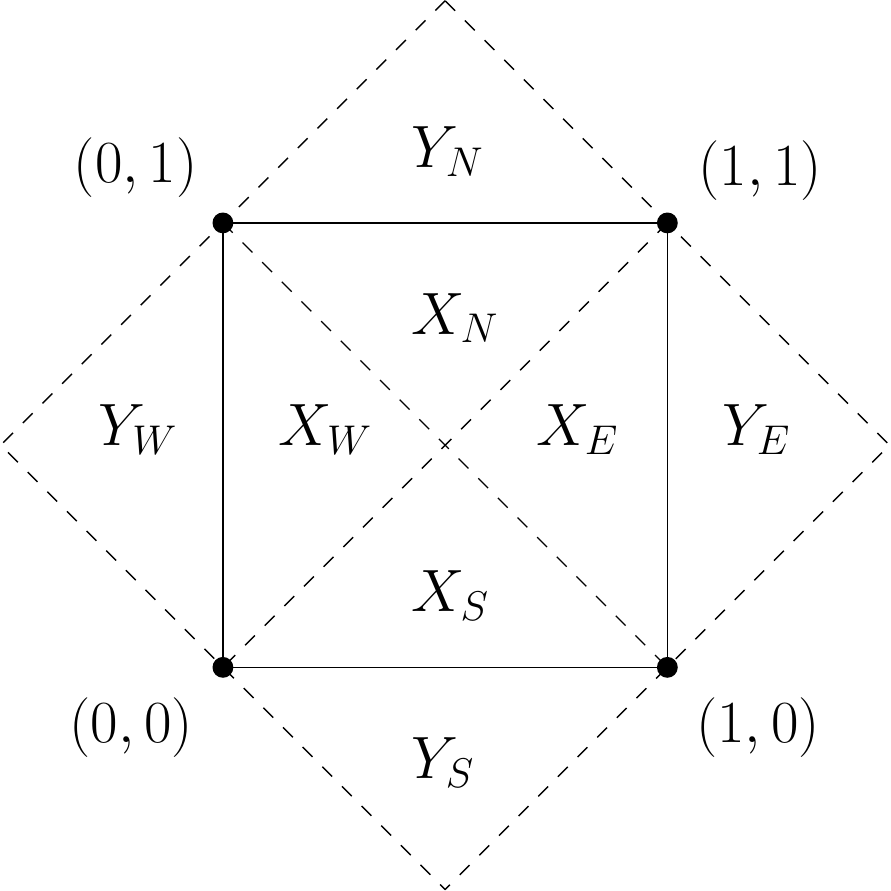}
\caption{New notation in and around face $F_{\tfrac{1}{2},\tfrac{1}{2}}$.}
\label{fig:XYrenamed}
\end{figure}

Let $S'=\mu_B(S)$. We will denote with a prime the corresponding $X$ coordinates of $S'$. For example $X_N'$ is the $X$ coordinate of $S'$ associated with the triangle in position $\left(\frac{1}{2},\frac{1}{2},N\right)$. We can express $X'_N$ and $Y'_N$ as functions of $X_D$ and $Y_D$ for $D\in\calD$:

\begin{theorem}
\label{thm:recurrenceformulas}
The following formulas hold :
\begin{gather}
Y'_N=Y_N \frac{\left(1-\frac{\big(1-X_W^{-1}\big)\big(1-Y_N^{-1}\big)}{\big(1-Y_W\big)\big(1-X_N\big)}\right)\left(1-\frac{\big(1-X_E^{-1}\big)\big(1-Y_N^{-1}\big)}{\big(1-Y_E\big)\big(1-X_N\big)}\right)}
{\left(1-\frac{\big(1-X_W\big)\big(1-Y_N\big)}{\big(1-Y_W^{-1}\big)\big(1-X_N^{-1}\big)}\right)\left(1-\frac{\big(1-X_E\big)\big(1-Y_N\big)}{\big(1-Y_E^{-1}\big)\big(1-X_N^{-1}\big)}\right)} \label{eq:newY} \\
X_N'= \frac{1-\frac{\big(1-X_N^{-1}\big)\big(1-Y_W^{-1}\big)\big(1-{Y'_N}^{-1}\big)}{\big(1-Y_N\big)\big(1-X_W\big)\big(1-Y'_W\big)}}
{1-\frac{\big(1-X_N\big)\big(1-Y_W\big)\big(1-Y'_N\big)}{\big(1-Y_N^{-1}\big)\big(1-X_W^{-1}\big)\big(1-{Y'_W}^{-1}\big)}}.
 \label{eq:newX}
\end{gather}

\end{theorem}

\begin{remark}
The quantity $Y'_W$ needed to compute $X'_N$ can be expressed as a function of $X_N,X_W,X_S,Y_N,Y_W$ and $Y_S$, with a formula similar to~\eqref{eq:newY}.
\end{remark}

\begin{proof}
We will use several times the following lemma, the proof of which is provided after the proof of the theorem :
\begin{lemma}
\label{lem:anglecomputation}
Let $\left(c_0,\ldots,c_k,d_0,\ldots,d_k\right)\in\left(0,2\pi\right)^{2k+2}$ and write $C_j=e^{ic_j}$ and $D_j=e^{id_j}$ for $0\leq j\leq k$. Let $p\in\left\{0,1\right\}$ and assume the following relations hold:
\begin{gather}
\frac{1}{2}\sum_{j=0}^k \left(c_j+d_j\right) \equiv p\pi \mod 2\pi \label{eq:sumangleslem} \\
\prod_{j=0}^k \sin\frac{c_j}{2}=\prod_{j=0}^k \sin\frac{d_j}{2} \label{eq:productsineslem}.
\end{gather}
Then we have
\begin{align}
C_0&=\frac{(-1)^p+\frac{\prod_{j=1}^k\left(1-D_j^{-1}\right)}{\prod_{j=1}^k\left(C_j-1\right)}}{(-1)^p+\frac{\prod_{j=1}^k\left(D_j-1\right)}{\prod_{j=1}^k\left(1-C_j^{-1}\right)}} \\
D_0&=\frac{(-1)^p+\frac{\prod_{j=1}^k\left(1-C_j^{-1}\right)}{\prod_{j=1}^k\left(D_j-1\right)}}{(-1)^p+\frac{\prod_{j=1}^k\left(C_j-1\right)}{\prod_{j=1}^k\left(1-D_j^{-1}\right)}}.
\end{align}
\end{lemma}

We rename the vertices and centers we will need, in order to have more compact notation. Define $A=S(0,1)$, $B=S(0,0)$, $C=S(1,0)$, $D=S(1,1)$, $A'=S'(0,1)$, $B'=S'(0,0)$, $C'=S'(1,0)$, $D'=S'(1,1)$, $O_1=O_{\tfrac{1}{2},\tfrac{3}{2}}^S$, $O_2=O_{-\tfrac{1}{2},\tfrac{1}{2}}^S$, $O_3=O_{\tfrac{1}{2},-\tfrac{1}{2}}^S$, $O_4=O_{\tfrac{3}{2},\tfrac{1}{2}}^S$, $O=O_{\tfrac{1}{2},\tfrac{1}{2}}^S$ and $O'=O_{\tfrac{1}{2},\tfrac{1}{2}}^{S'}$.

We begin with the following observation about half-angles of kites. Let $EFGH$ be a kite, i.e. a quadrilateral such that $EF=EH$ and $GF=GH$. Write $\alpha=\angle FEH$ and $\gamma=\angle HGF$ with $(\alpha,\gamma)\in\left[0,2\pi\right)^2$, hence $\left(\tfrac{\alpha}{2},\tfrac{\gamma}{2}\right)\in\left[0,\pi\right)^2$. If $EFGH$ is positively oriented, then $\angle FEG \equiv \tfrac{\alpha}{2} \mod 2\pi$ and $\angle EGF \equiv \tfrac{\gamma}{2} \mod 2\pi$, while if $EFGH$ is negatively oriented, then $\angle FEG \equiv \tfrac{\alpha}{2} +\pi \mod 2\pi$ and $\angle EGF \equiv \tfrac{\gamma}{2} +\pi \mod 2\pi$. In any case, $\angle FEG + \angle EGF \equiv \tfrac{\alpha+\gamma}{2} \mod 2\pi$.

\begin{figure}[htbp]
\centering
\subfloat[Kites before mutation.]{\label{fig:oldkites}\includegraphics[height=1.5in]{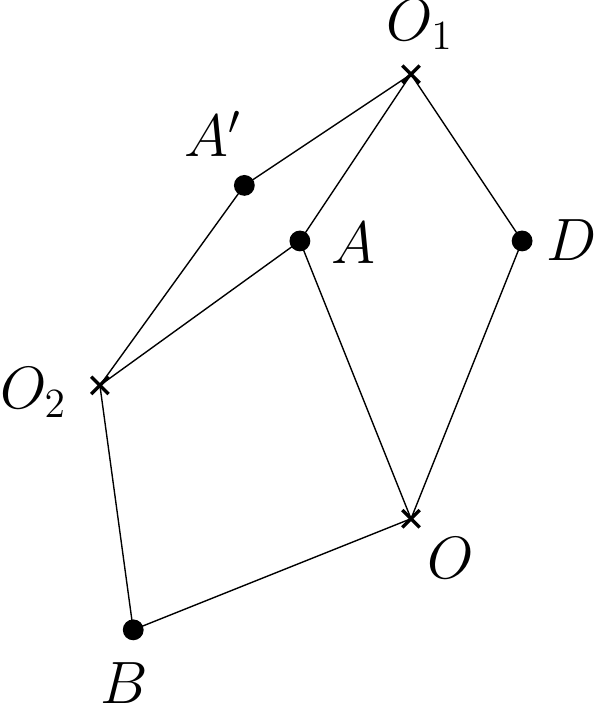}}
\hspace{2em}
\subfloat[Kites after mutation.]{\label{fig:newkites}\includegraphics[height=1.5in]{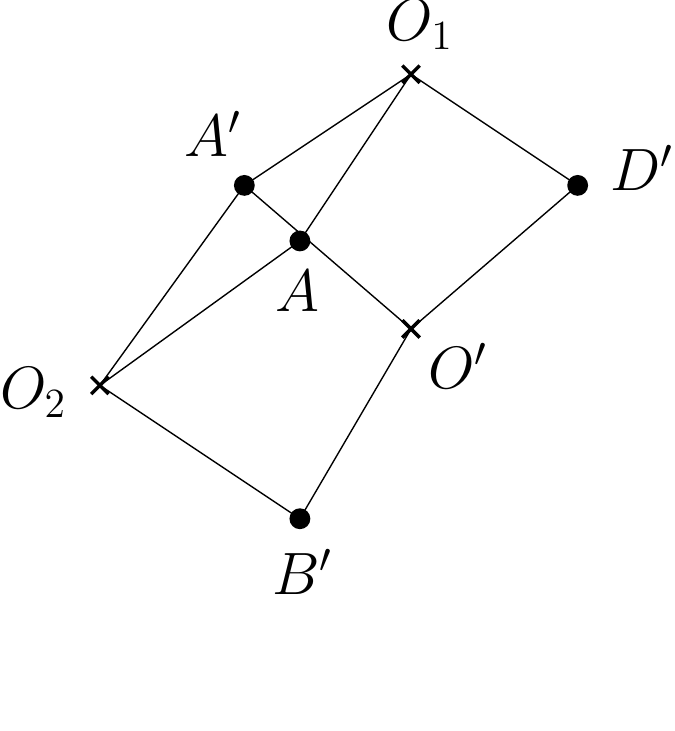}}
\caption{The two triples of kites used to prove formulas~\eqref{eq:sumhalfanglesS} and~\eqref{eq:sumhalfanglesT}.}
\label{fig:oldandnewkites}
\end{figure}

We now prove formula~\eqref{eq:newY}. Using the previous observation in each of the three kites $O_1A'O_2A$, $O_2BOA$ and $ODO_1A$, we obtain the following relations :
\begin{align}
\frac{\angle A'O_1A+\angle AO_2A'}{2}&\equiv \angle O_2O_1A + \angle AO_2O_1 &&\mod 2\pi \\
\frac{\angle BO_2A +\angle AOB}{2}&\equiv \angle OO_2A + \angle AOO_2 &&\mod 2\pi \\
\frac{\angle DOA + \angle AO_1D}{2}&\equiv \angle O_1OA + \angle AO_1O &&\mod 2\pi.
\end{align}

Thus
\begin{equation}
\label{eq:sumhalfanglesS}
\frac{\angle A'O_1A+\angle AO_2A'+\angle BO_2A +\angle AOB+\angle DOA + \angle AO_1D}{2}\equiv \pi \mod 2\pi.
\end{equation}
Furthermore, writing ratios of radii of the circles $\calC_{\frac{1}{2},\frac{1}{2}}^S$, $\calC_{-\frac{1}{2},\frac{1}{2}}^S$ and $\calC_{\frac{1}{2},\frac{3}{2}}^S$ as ratios of sines of half-angles, we obtain
\[
\frac{\sin \frac{\angle AO_2A'}{2}\sin \frac{\angle AOB}{2}\sin \frac{\angle AO_1D}{2}}{\sin \frac{\angle A'O_1A}{2}\sin \frac{\angle BO_2A}{2}\sin \frac{\angle DOA}{2}}=1.
\]
Applying Lemma~\ref{lem:anglecomputation} with $p=1$, we deduce that
\[
e^{i\angle A'O_1A}=
\frac{\left(1-\frac{\big(1-X_W^{-1}\big)\big(1-Y_N^{-1}\big)}{\big(1-Y_W\big)\big(1-X_N\big)}\right)}
{\left(1-\frac{\big(1-X_W\big)\big(1-Y_N\big)}{\big(1-Y_W^{-1}\big)\big(1-X_N^{-1}\big)}\right)}.
\]
Similarly,
\[
e^{i\angle DO_1D'}=
\frac{\left(1-\frac{\big(1-X_E^{-1}\big)\big(1-Y_N^{-1}\big)}{\big(1-Y_E\big)\big(1-X_N\big)}\right)}
{\left(1-\frac{\big(1-X_E\big)\big(1-Y_N\big)}{\big(1-Y_E^{-1}\big)\big(1-X_N^{-1}\big)}\right)}.
\]
Since circumcircles of white faces are not changed by $\mu_B$, we have $O_1=O_{\tfrac{1}{2},\tfrac{3}{2}}^{S'}$ and we finally obtain~\eqref{eq:newY} by noting that
\[
Y'_N = e^{i\left(\angle A'O_1A + \angle AO_1D + \angle DO_1D'\right)}.
\]

By a proof similar to the one of~\eqref{eq:sumhalfanglesS}, we have
\begin{multline}
\label{eq:sumhalfanglesT}
\frac{1}{2}\bigg(\angle AO_1A'+\angle A'O_2A+\angle B'O_2A' +\angle A'O'B' \\
+\angle D'O'A' + \angle A'O_1D'\bigg)\equiv \pi \mod 2\pi.
\end{multline}
Summing~\eqref{eq:sumhalfanglesS} and~\eqref{eq:sumhalfanglesT}, we get
\begin{multline}
\frac{1}{2}\bigg(\angle BO_2A +\angle AOB+\angle DOA + \angle AO_1D +\angle B'O_2A' +\angle A'O'B' \\
+\angle D'O'A' + \angle A'O_1D'\bigg)\equiv 0\mod 2\pi.
\end{multline}
Furthermore, rewriting the formula
\[
\frac{R_{\frac{1}{2},\frac{3}{2}}^S}{R_{\frac{1}{2},\frac{1}{2}}^S}
\frac{R_{\frac{1}{2},\frac{1}{2}}^S}{R_{-\frac{1}{2},\frac{1}{2}}^S}
\frac{R_{-\frac{1}{2},\frac{1}{2}}^{S'}}{R_{\frac{1}{2},\frac{1}{2}}^{S'}}
\frac{R_{\frac{1}{2},\frac{1}{2}}^{S'}}{R_{\frac{1}{2},\frac{3}{2}}^{S'}}=1
\]
in terms of sines of half-angles, we obtain
\[
\frac{\sin\frac{\angle DOA}{2}\sin\frac{\angle BO_2A}{2}\sin\frac{\angle A'O'B' }{2}\sin\frac{\angle A'O_1D' }{2}}{\sin\frac{\angle AO_1D}{2}\sin\frac{\angle AOB}{2}\sin\frac{\angle B'O_2A'}{2}\sin\frac{\angle D'O'A'}{2}}=1.
\]
Formula~\eqref{eq:newX} then follows from Lemma~\ref{lem:anglecomputation} with $p=0$.
\end{proof}

We now prove Lemma~\ref{lem:anglecomputation}.
\begin{proof}
Rewrite~\eqref{eq:productsineslem} as
\[
\prod_{j=0}^k e^{i\frac{c_j}{2}}\left(1-C_j^{-1}\right)=\prod_{j=0}^k e^{-i\frac{d_j}{2}}\left(D_j-1\right).
\]
Using~\eqref{eq:sumangleslem}, we observe that
\[
\prod_{j=0}^k e^{i\frac{c_j+d_j}{2}}=(-1)^p,
\]
hence
\begin{equation}
\label{eq:linsys}
\prod_{j=0}^k \left(1-C_j^{-1}\right)=(-1)^p\prod_{j=0}^k \left(D_j-1\right).
\end{equation}
From~\eqref{eq:sumangleslem} we obtain that
\[
D_0=C_0^{-1}\prod_{j=1}^k C_j^{-1}D_j^{-1}.
\]
Substituting this in~\eqref{eq:linsys}, we obtain the formula for $C_0$. The formula for $D_0$ follows by symmetry.
\end{proof}

\section{The isoradial case}
\label{sec:isoradial}

As an application of the recurrence formulas of Theorem~\ref{thm:recurrenceformulas}, we show that the dynamics is trivial on some subset of $\calS_{m,n}^s$, which is equal to the whole of $\calS_{m,n}^s$ when $n=1$ and $s\in\left\{1,m-1\right\}$.

A circle pattern is said to be \emph{isoradial} if all the circles have a common radius. Isoradial circle patterns form a well-studied class of graphs, see for example~\cite{Duffin,Kenyon1,KS}. Denote by $\calI_{m,n}^s$ the subset of $\calS_{m,n}^s$ formed of isoradial circle patterns. We have the following characterization of isoradial patterns in terms of the $\phi$ angles :

\begin{lemma}
\label{lem:isoradialphi}
A circle pattern $S\in\calS_{m,n}^s$ is isoradial if and only if the following two statements hold :
\begin{align}
\forall (i,j)\in\Gamma, \ &\phi_{i,j}^E=\phi_{i+1,j}^W \label{eq:isoradialphi1} \\
\forall (i,j)\in\Gamma, \ &\phi_{i,j}^N=\phi_{i,j+1}^S. \label{eq:isoradialphi2}
\end{align}
\end{lemma}

\begin{proof}
Clearly, a circle pattern is isoradial if and only if all the kites formed by two adjacent vertices and the two circumcenters of the two faces containing these two vertices are rhombi (see Figure~\ref{fig:rhombus}). In terms of the $\phi$ variables, the two $\phi$'s present in a kite must be equal.
\end{proof}

\begin{figure}[htpb]
\centering
\includegraphics[height=2in]{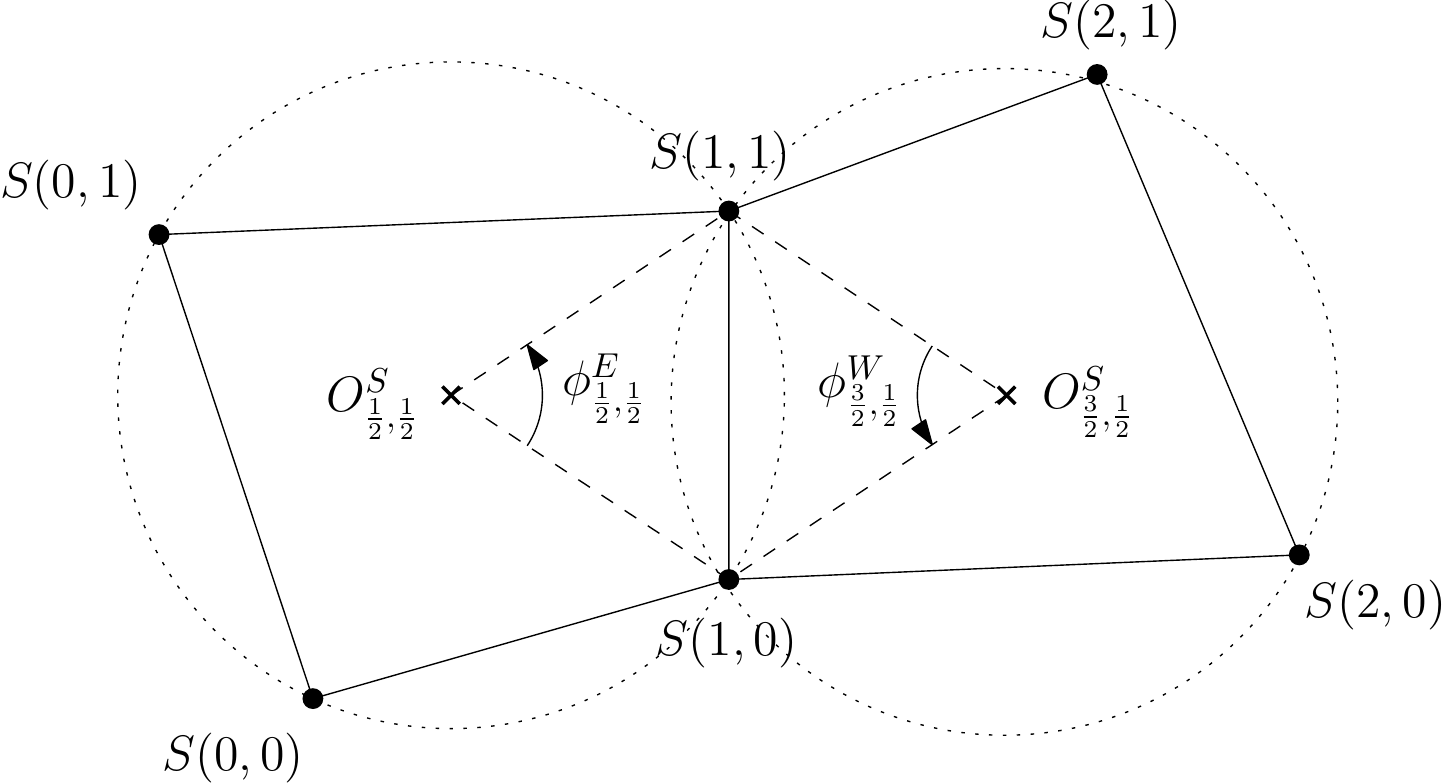}
\caption{The kite around the edge connecting $S(1,0)$ to $S(1,1)$ is a rhombus when the two circles have the same radius.}
\label{fig:rhombus}
\end{figure}

As a consequence of Lemma~\ref{lem:isoradialphi}, we observe that for isoradial patterns, $\phi$ angles are naturally associated with edges (the $\phi$ angle of an edge is the $\phi$ angle of the rhombus surrounding it). Using this characterization of isoradial patterns, we deduce the following :

\begin{proposition}
\label{prop:isoradialstable}
The subset $\calI_{m,n}^s$ is stable under Miquel dynamics. Furthermore, Miquel dynamics acts on the $\phi$ variables by a shift : whenever a face mutates, the $\phi$ variables associated with its top and bottom edges (resp. left and right edges) are exchanged.
\end{proposition}

\begin{proof}
We use the notation of Theorem~\ref{thm:recurrenceformulas} and assume that for any $D\in\calD$, we have $X_D=Y_D$. Then it follows from equations~\eqref{eq:newY} and~\eqref{eq:newX} and from $X_NX_WX_SX_E=1$ that $Y_N'=X_N'=X_S$. By symmetry, we also have $Y_W'=X_W'=X_E$, $Y_S'=X_S'=X_N$ and $Y_E'=X_E'=X_W$.
\end{proof}

In particular, Proposition~\ref{prop:isoradialstable} implies that Miquel dynamics is periodic of period
\[
\lcm\left(m,\tfrac{mn}{\gcd(m,s)}\right).
\]
Most of the time, $\calI_{m,n}^s$ is a strict subset of $\calS_{m,n}^s$ but in some cases there is equality :

\begin{lemma}
\label{lem:mby1}
If $n=1$ and $s\in\left\{1,m-1\right\}$, then $\calI_{m,n}^s=\calS_{m,n}^s$.
\end{lemma}

\begin{proof}
We assume $s=1$, the case $s=m-1$ is similar. Fix $S\in\calS_{m,1}^1$. We will show that for any $1\leq x\leq m-1$, the triangles $S(x-1,0)S(x,0)S(x,1)$ and $S(x+1,1)S(x,1)S(x,0)$ are isometric, hence their circumcircles have the same radius. Thus the circumcircles of all the faces have the same radius. Let $P$ be the midpoint of the segment $\left[S(x,0)S(x,1)\right]$ and let $\sigma_P$ denote the reflection across $P$. Given that $S\in\calS_{m,1}^1$, we have $\overrightarrow{S(x-1,0)S(x,0)}=\overrightarrow{S(x,1)S(x+1,1)}$. Since we also have $\overrightarrow{PS(x,1)}+\overrightarrow{PS(x,0)}=\overrightarrow{0}$, we deduce that $\overrightarrow{PS(x+1,1)}+\overrightarrow{PS(x-1,0)}=\overrightarrow{0}$. So $\sigma_P$ maps the triangle $S(x-1,0)S(x,0)S(x,1)$ to the triangle $S(x+1,1)S(x,1)S(x,0)$.
\end{proof}

By Proposition~\ref{prop:isoradialstable}, Miquel dynamics is trivial on $\calI_{m,n}^s$. So by Lemma~\ref{lem:mby1}, we have a complete understanding of Miquel dynamics when $n=1$ and $s\in\left\{1,m-1\right\}$. This includes the cases $(m,n)=(2,1)$ and $(m,n)=(4,1)$.

\section{Some conserved quantities}
\label{sec:conserved}

In this section we discuss two types of conserved quantities.

\subsection{Monodromy invariants}
\label{subsec:monodromyconserved}

The first conserved quantities are related to the monodromies of a biperiodic circle pattern. Let $S\in \calS_{m,n}^s$ with monodromy $\vec{u}=(u_x,u_y)$ (resp. $\vec{v}=(v_x,v_y)$) in the direction $(m,0)$ (resp. $(s,n)$). Then the iterates of $S$ under Miquel dynamics have the same monodromies, because the elementary operations used by the dynamics (reflection and ``taking the circumcenter of three points'') commute with translations. If we consider instead the equivalence class of $S$ in $\calM_{m,n}^s$, then the pair of vectors $\left(\vec{u},\vec{v}\right)$ is only well-defined up to a common similarity acting on both of them, hence the conserved quantities are the real and imaginary parts of $\tfrac{v_x+iv_y}{u_x+iu_y}$ i.e. the pair $\left(\tfrac{u_xv_x+u_yv_y}{u_x^2+u_y^2},\tfrac{u_xv_y-u_yv_x}{u_x^2+u_y^2}\right)$. This holds for generic elements of $\calM_{m,n}^s$, for which $\vec{u}\neq\vec{0}$.

\subsection{Intersection angles along dual loops}
\label{subsec:anglesconserved}

The second type of conserved quantities is the sum of the exterior intersection angles along a loop in the dual graph.

Fix $m,n$ and $s$. We denote by $G$ the graph obtained by projecting $\Z^2$ down to the torus $\T$ with a fundamental domain of size $m$ by $n$ and glued with a horizontal shift of $s$. We construct $G^*$ the dual graph of $G$ by creating a dual vertex for each face of the primal graph $G$ and by connecting two dual vertices by a dual edge if the corresponding two faces in the primal graph share an edge. A dual vertex is colored black (resp. white) if the corresponding face in $G$ is colored black (resp. white), see Figure~\ref{fig:dualgraph}.

Fix $S\in\calS_{m,n}^s$. To each undirected edge $e$ of $G^*$, we associate the exterior intersection angle $\theta_S(e)\in[0,2\pi)$  of the circles associated with the two endpoints of $e$. If $e$ is a horizontal edge, connecting $(i,j)$ to $(i+1,j)$, set
\[
\theta_S(e)=\angle O_{i+1,j}^S S\left(i+\frac{1}{2},j-\frac{1}{2}\right) O_{i,j}^S.
\]
If $e$ is a vertical edge, connecting $(i,j)$ to $(i,j+1)$, set
\[
\theta_S(e)=\angle O_{i,j}^S S\left(i-\frac{1}{2},j+\frac{1}{2}\right) O_{i,j+1}^S.
\]
Observe that these angles can be expressed using the $\phi$ angles associated with $S$ : if $e$ connects $(i,j)$ to $(i+1,j)$, then
\[
\theta_S(e)=\pi-\frac{1}{2}\left(\phi_{i,j}^E+\phi_{i+1,j}^W\right)
\]
while if $e$ connects $(i,j)$ to $(i,j+1)$, then
\[
\theta_S(e)=\pi-\frac{1}{2}\left(\phi_{i,j}^N+\phi_{i,j+1}^S\right).
\]
In particular, it follows from equation~\eqref{eq:faceflat} that for any vertex $(x,y)$ in $G$,
\begin{equation}
\label{eq:thetaflat}
\sum_{e\in E_{(x,y)}} \theta_S(e) \equiv 0 \mod 2\pi,
\end{equation}
where $E_{(x,y)}$ denotes the set of four dual edges surrounding the dual face corresponding to the vertex $(x,y)$.

We also define a sign function $\epsilon$ on the set of directed edges of $G^*$. For any directed edge $\vec{e}$ of $G^*$, we set $\epsilon(\vec{e})$ to be $+1$ if $\vec{e}$ is vertical and goes from a black vertex to a white vertex or if $\vec{e}$ is horizontal and goes from a white vertex to a black vertex. Otherwise, we set $\epsilon(\vec{e})$ to be $-1$.

\begin{figure}[htpb]
\centering
\includegraphics[height=1.2in]{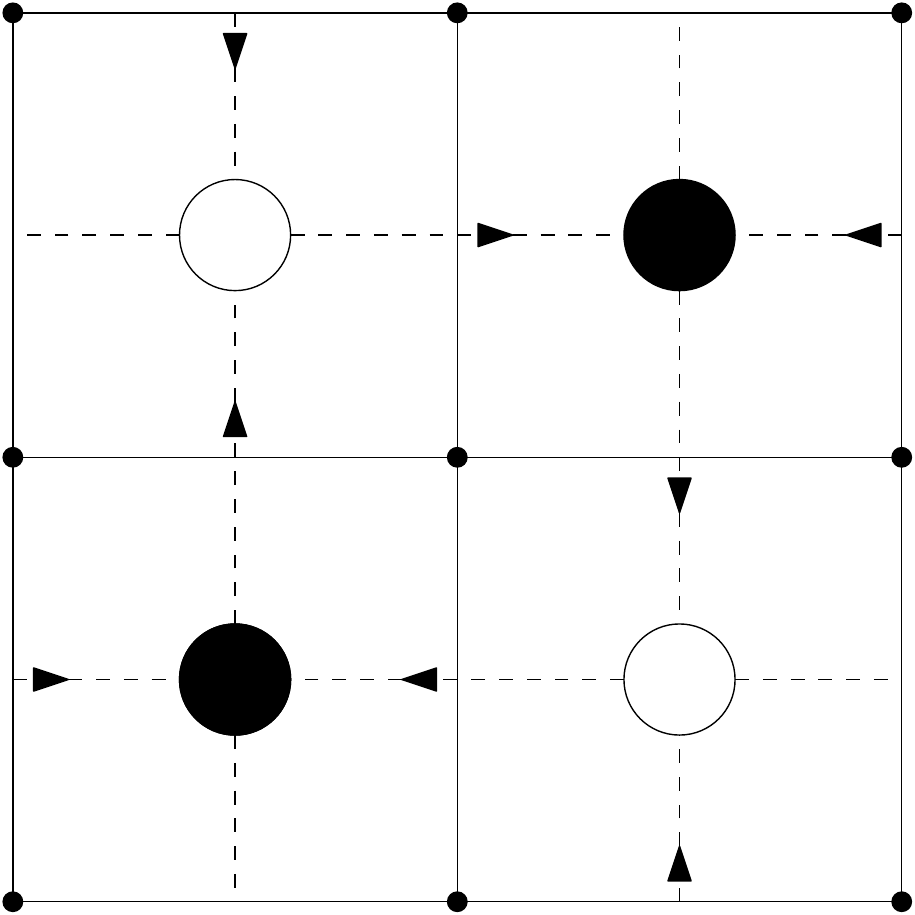}
\caption{The dual graph $G^*$ when $m=n=2$ and $s=0$. Its edges are represented by dashed lines and carry the orientation corresponding to $\epsilon=+1$.}
\label{fig:dualgraph}
\end{figure}

A directed loop $l$ in $G^*$ is given by a finite sequence $(\overrightarrow{e_1},\ldots,\overrightarrow{e_{2k}})$ of directed dual edges such that the head of $\overrightarrow{e_h}$ equals the tail of $\overrightarrow{e_{h+1}}$ for all $1\leq h\leq 2k-1$ and the head of $\overrightarrow{e_{2k}}$ equals the tail of $\overrightarrow{e_1}$.
We define the signed sum of the exterior intersection angles associated with a dual loop. Given a directed loop $l$ in $G^*$, we set
\begin{equation}
\tilde{\gamma}_S(l)=\sum_{\vec{e}\in l} \epsilon(\vec{e}) \theta_S(e) \mod 2\pi,
\end{equation}
where $e$ denotes the undirected edge in $G^*$ corresponding to $\vec{e}$.

When $l$ is the boundary of a dual face (directed loop of length four), equation~\eqref{eq:thetaflat} implies that $\tilde{\gamma}_S(l)=0$. More generally, for any directed loop $l$ having zero homology in $H_1(\T,\Z)$, we have $\tilde{\gamma}_S(l)=0$. Thus, there exists a unique group homomorphism
$\gamma_S:H_1(\T,\Z)\rightarrow \R/(2\pi\Z)$ such that for any directed loop $l$ in $G^*$ with homology class $\left[l\right]$ in $H_1(\T,\Z)$,
\[
\gamma_S(\left[l\right])=\tilde{\gamma}_S(l).
\]
This homomorphism $\gamma_S$ is (essentially) conserved by Miquel dynamics:

\begin{theorem}
\label{thm:gammaconserved}
For any $S\in \calS_{m,n}^s$, we have
\begin{equation}
\gamma_{\mu_B(S)}=\gamma_{\mu_W(S)}=-\gamma_S.
\end{equation}
\end{theorem}

\begin{proof}
It suffices to prove that for any directed dual loop $l$, we have $\tilde{\gamma}_{\mu_B(S)}(l)=-\tilde{\gamma}_S(l)$ and $\tilde{\gamma}_{\mu_W(S)}(l)=-\tilde{\gamma}_S(l)$. We will prove the first statement.
Let $l$ be a directed dual loop. Up to a cyclic permutation of its edges, we may write $l$ as a sequence of consecutive directed dual edges $(\overrightarrow{e_1},\ldots,\overrightarrow{e_{2k}})$ such that for any $1\leq h\leq k$, the head of $\overrightarrow{e_{2h}}$ is a white dual vertex. We will show that for any $1\leq h\leq k$, we have
\begin{multline}
\epsilon(\overrightarrow{e_{2h-1}})\theta_S(e_{2h-1})+\epsilon(\overrightarrow{e_{2h}})\theta_S(e_{2h}) + \\
\epsilon(\overrightarrow{e_{2h-1}})\theta_{\mu_B(S)}(e_{2h-1})+\epsilon(\overrightarrow{e_{2h}})\theta_{\mu_B(S)}(e_{2h}) \equiv 0 \mod 2 \pi.
\end{multline}
Summing over $h$ yields $\tilde{\gamma}_S(l)+\tilde{\gamma}_{\mu_B(S)}(l)=0$. We distinguish two cases.

\emph{Case 1.}
We first assume that $\epsilon(\overrightarrow{e_{2h-1}})=\epsilon(\overrightarrow{e_{2h}})$. We need to show that
\begin{equation}
\label{eq:case1}
\theta_S(e_{2h-1})+\theta_S(e_{2h})+\theta_{\mu_B(S)}(e_{2h-1})+\theta_{\mu_B(S)}(e_{2h}) \equiv 0 \mod 2 \pi.
\end{equation}
To simplify notations, we assume (without loss of generality) that $h=1$, that $\overrightarrow{e_1}$ connects $\left(-\tfrac{1}{2},\tfrac{1}{2}\right)$ to $\left(\tfrac{1}{2},\tfrac{1}{2}\right)$ and $\overrightarrow{e_2}$ connects $\left(\tfrac{1}{2},\tfrac{1}{2}\right)$ to $\left(\tfrac{1}{2},\tfrac{3}{2}\right)$. Write $S'=\mu_B(S)$. Then we have the following formulas for the $\theta$ angles of $S$ and $S'$ on edges $e_1$ and $e_2$:
\begin{align*}
\theta_S(e_1)&=\angle O_{-\frac{1}{2},\frac{1}{2}}^S S(0,1) O_{\frac{1}{2},\frac{1}{2}}^S \\
\theta_S(e_2)&=\angle O_{\frac{1}{2},\frac{1}{2}}^S S(0,1) O_{\frac{1}{2},\frac{3}{2}}^S \\
\theta_{S'}(e_1)&=\angle O_{-\frac{1}{2},\frac{1}{2}}^{S'} S'(0,1) O_{\frac{1}{2},\frac{1}{2}}^{S'} \\
\theta_{S'}(e_2)&=\angle O_{\frac{1}{2},\frac{1}{2}}^{S'} S'(0,1) O_{\frac{1}{2},\frac{3}{2}}^{S'}
\end{align*}

\begin{figure}[htbp]
\centering
\subfloat[Angles in the pattern $S$.]{\label{fig:gammaconserved1}\includegraphics[height=1.2in]{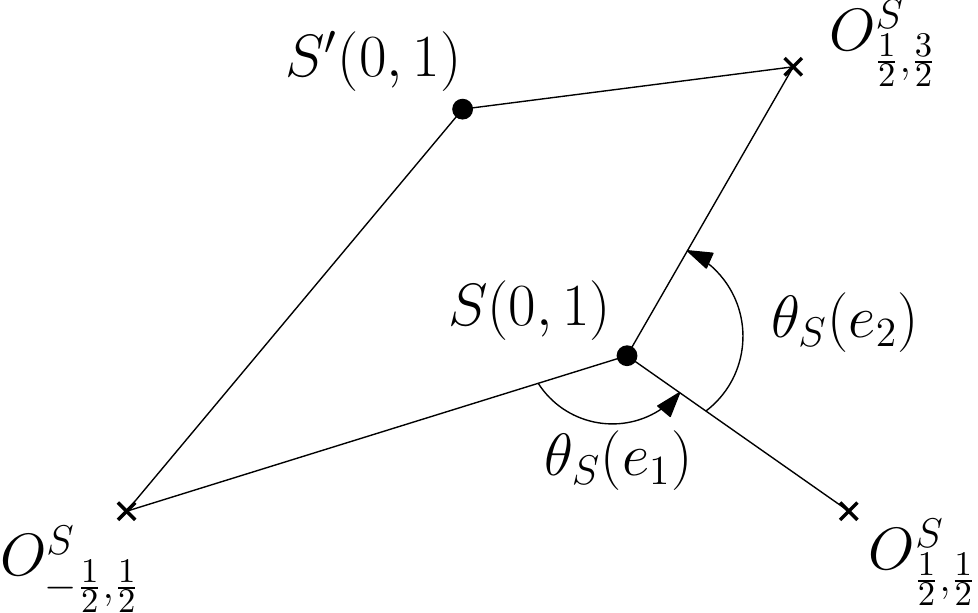}}
\hspace{2em}
\subfloat[Angles in the pattern $S'$.]{\label{fig:gammaconserved2}\includegraphics[height=1.2in]{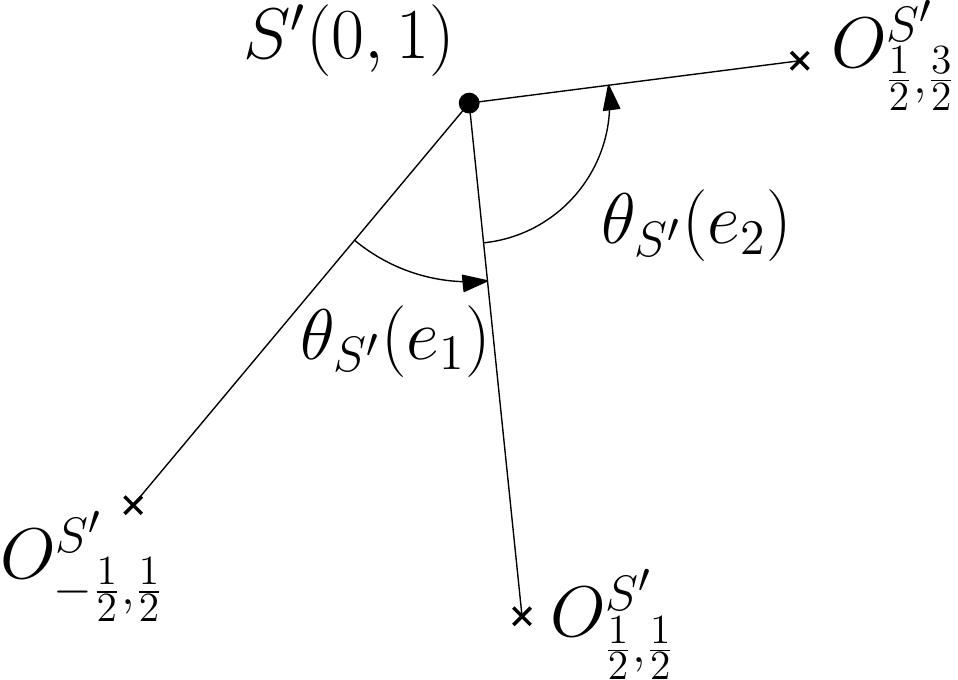}}
\caption{Equation~\eqref{eq:case1} follows from opposite angles being equal in the kite $O_{-\frac{1}{2},\frac{1}{2}}^S S(0,1) O_{\frac{1}{2},\frac{3}{2}}^S S'(0,1)$.}
\label{fig:gammaconserved}
\end{figure}

Since white circles are invariant under black mutation $\mu_B$, $O_{-\frac{1}{2},\frac{1}{2}}^S=O_{-\frac{1}{2},\frac{1}{2}}^{S'}$ and $O_{\frac{1}{2},\frac{3}{2}}^S=O_{\frac{1}{2},\frac{3}{2}}^{S'}$. Thus
\begin{multline*}
\theta_S(e_1)+\theta_S(e_2)+\theta_{S'}(e_1)+\theta_{S'}(e_2) \\
\equiv\angle O_{-\frac{1}{2},\frac{1}{2}}^S S(0,1) O_{\frac{1}{2},\frac{3}{2}}^S + \angle O_{-\frac{1}{2},\frac{1}{2}}^S S'(0,1) O_{\frac{1}{2},\frac{3}{2}}^S \equiv 0 \mod 2\pi,
\end{multline*}
because $S'(0,1)$ is the reflection of $S(0,1)$ through the line $\left(O_{-\frac{1}{2},\frac{1}{2}}^S,O_{\frac{1}{2},\frac{3}{2}}^S\right)$.

\emph{Case 2.} Now assume we have $\epsilon(\overrightarrow{e_{2h-1}})=-\epsilon(\overrightarrow{e_{2h}})$. In this case, we need to show that
\begin{equation}
\label{eq:case2}
\theta_S(e_{2h-1})-\theta_S(e_{2h})+\theta_{\mu_B(S)}(e_{2h-1})-\theta_{\mu_B(S)}(e_{2h}) \equiv 0 \mod 2 \pi.
\end{equation}
We can find two directed edges $\overrightarrow{e_{2h-1}'}$ and $\overrightarrow{e_{2h}'}$ such that the following conditions hold :
\begin{enumerate}
\item $\overrightarrow{e_{2h-1}'}$ and $\overrightarrow{e_{2h}'}$ have the same underlying undirected edge, but have opposite orientations ;
\item The head of $\overrightarrow{e_{2h-1}}$ is equal to the tail of $\overrightarrow{e_{2h}'}$ ;
\item $\epsilon(\overrightarrow{e_{2h-1}})=\epsilon(\overrightarrow{e_{2h}'})$.
\end{enumerate}
It follows from these conditions that $\epsilon(\overrightarrow{e_{2h-1}'})=\epsilon(\overrightarrow{e_{2h}})$. See Figure~\ref{fig:case2} for an example.

\begin{figure}[htpb]
\centering
\includegraphics[height=1.2in]{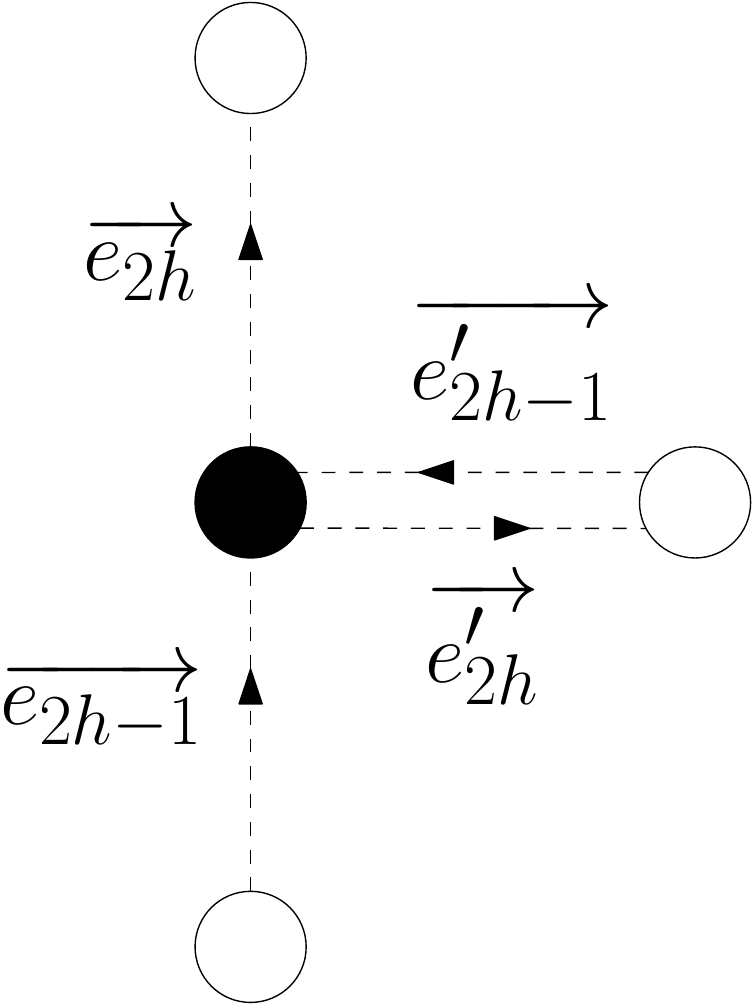}
\caption{A possible choice of $\protect\overrightarrow{e_{2h-1}'}$ and $\protect\overrightarrow{e_{2h}'}$ in case 2. They correspond to the two directions of a common dual edge.}
\label{fig:case2}
\end{figure}

Using case 1, we deduce that
\begin{align}
\theta_S(e_{2h-1})+\theta_S(e_{2h}')+\theta_{\mu_B(S)}(e_{2h-1})+\theta_{\mu_B(S)}(e_{2h}') &\equiv 0 \mod 2 \pi \label{eq:case2-1} \\
\theta_S(e_{2h-1}')+\theta_S(e_{2h})+\theta_{\mu_B(S)}(e_{2h-1}')+\theta_{\mu_B(S)}(e_{2h}) &\equiv 0 \mod 2 \pi \label{eq:case2-2}
\end{align}
Since $e_{2h-1}'=e_{2h}'$, taking the difference of~\eqref{eq:case2-1} and~\eqref{eq:case2-2} gives~\eqref{eq:case2}.
\end{proof}

\begin{remark}
\label{rem:indepconservedquantitites}
The group homomorphism $\gamma_S$ is determined by the images of two loops generating $H_1(\T,\Z)$.
Thus, Theorem~\ref{thm:gammaconserved} provides at most two independent conserved quantities. We will see in Section~\ref{sec:twobytwo} that in the case $m=n=2$ and $s=0$, the images of two loops generating $H_1(\T,\Z)$ are indeed two independent conserved quantities.
\end{remark}

\section{The $2\times2$ case}
\label{sec:twobytwo}

In this section we provide a description of the trajectory of the vertices in the case when $m=n=2$ (which implies $s=0$).

\subsection{Notation and construction}

Fix a circle pattern $S\in\calS_{2,2}^0$. For simplicity, we will rename the vertices and centers of $S$ as follows : $A=S(0,0)$, $B=S(1,0)$, $C=S(2,0)$, $D=S(0,1)$, $E=S(1,1)$, $F=S(2,1)$, $G=S(0,2)$, $H=S(1,2)$ and $I=S(2,2)$ for the vertices, $\Omega_1=O_{\frac{1}{2},\frac{1}{2}}^{S}$, $\Omega_2=O_{\frac{1}{2},\frac{3}{2}}^{S}$, $\Omega_3=O_{\frac{3}{2},\frac{1}{2}}^{S}$ and $\Omega_4=O_{\frac{3}{2},\frac{3}{2}}^{S}$ for the centers (see Figure~\ref{fig:22notation}). We set $S_0=S$ and for any $t\in\Z$, we denote by $S_t\in\calS_{2,2}^0$ the time $t$ image of $S_0$ under Miquel dynamics. We define $\tilde{S}_t\in\calS_{2,2}^0$ to be the image of $S_t$ under the translation of vector $\overrightarrow{S_t(0,0)S_0(0,0)}$. Vertices of $\tilde{S}_t$ will be denoted by $A_t,\ldots,I_t$. For example, $E_t=\tilde{S}_t(1,1)$.

\begin{figure}[htpb]
\centering
\includegraphics[height=2.5in]{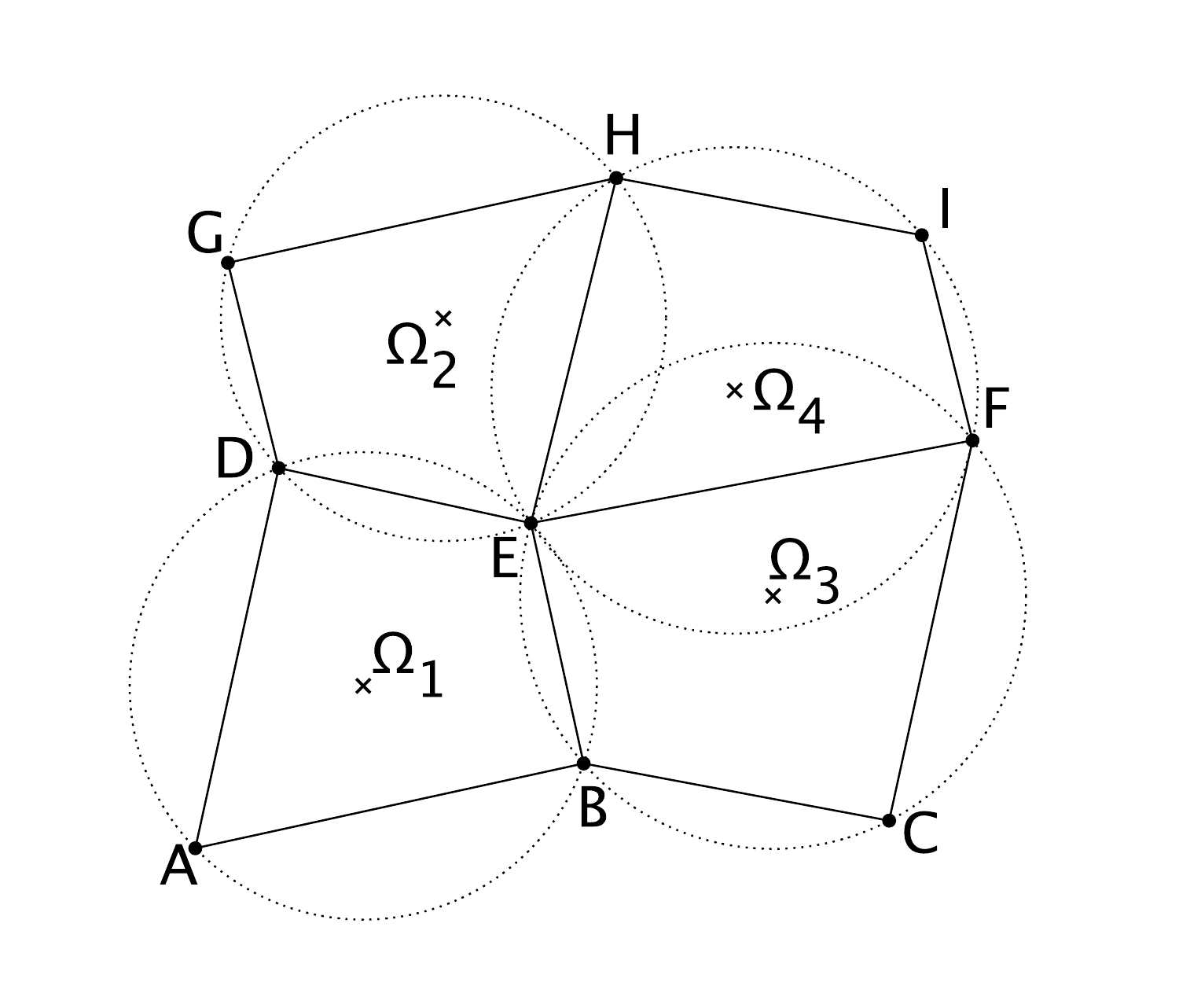}
\caption{Notation for $2\times2$ circle patterns.}
\label{fig:22notation}
\end{figure}

By construction, $A_t$ is independent of $t$. We denote by $\vec{u}$ (resp. $\vec{v}$) the monodromy of $S$ in the direction $(2,0)$ (resp. $(0,2)$). Since the monodromies are preserved by Miquel dynamics (see Subsection~\ref{subsec:monodromyconserved}), this implies that for any $t\in\Z$, $\overrightarrow{A_tC_t}=\overrightarrow{D_tF_t}=\overrightarrow{G_tI_t}=\vec{u}$, that $\overrightarrow{A_tG_t}=\overrightarrow{B_tH_t}=\overrightarrow{C_tI_t}=\vec{v}$ and that the positions of the points $C_t$, $G_t$ and $I_t$ are independent of $t$.

We now provide a practical way to construct circle patterns in the $2\times2$ case. Start with $B,D,F$ and $H$ four points in generic position in the plane. Set $\vec{u}=\overrightarrow{DF}$ and $\vec{v}=\overrightarrow{BH}$. Pick any point $E$ in the plane. Construct $\calC_1$, $\calC_2$, $\calC_3$ and $\calC_4$ to be respectively the circumcircles of the triangles $BDE$, $DEH$, $BEF$ and $EFH$. Construct $A$ as the other intersection point of $\calC_1$ with the circle obtained by translating $\calC_2$ along vector $-\vec{v}$ (the first intersection point being $B$). Set $C=A+\vec{u}$, $G=A+\vec{v}$ and $I=C+\vec{v}$. Then the following proposition characterizes when this procedure gives a biperiodic circle pattern $S$ such that $S(0,0)=A$, $S(1,0)=B$, $S(2,0)=C$, $S(0,1)=D$, $S(1,1)=E$, $S(2,1)=F$, $S(0,2)=G$, $S(1,2)=H$ and $S(2,2)=I$.

\begin{proposition}
\label{prop:hyperbola}
The above construction can be extended to an SGCP $S\in\calS_{2,2}^0$ with monodromy $\vec{u}$ (resp. $\vec{v}$) in the direction $(2,0)$ (resp. $(0,2)$) if and only if $E$ lies on the equilateral hyperbola going through $B,D,F$ and $H$. In the case when $(BH)\perp(DF)$, the condition becomes that $E$ must lie on $(BH)\cup(DF)$.
\end{proposition}

\begin{proof}
With the above construction, the faces $ABED$ and $DEHG$ are automatically cyclic. Recall that four points are concyclic if and only if their cross-ratio is real. Using a computer algebra program such as Mathematica, we compute the imaginary parts of the cross-ratios $\left[B,C,E,F\right]$ and $\left[E,F,H,I\right]$ and observe that they vanish if and only if $E$ lies on the equilateral hyperbola going through $B,D,F$ and $H$ (or on $(BH)\cup(DF)$ when $(BH)\perp(DF)$).
\end{proof}

\subsection{Shape of the fundamental domain}

We now describe how the conserved quantities found in Subsection~\ref{subsec:anglesconserved} (signed sums of intersection angles along dual loops) have a simple interpretation in the $2\times2$ case.

\begin{lemma}
\label{lem:22angles}
For any $t\in\Z$, we have
\begin{align}
\angle A_tD_tG_t & \equiv \angle H_tE_tB_t \mod 2\pi \label{eq:verticalanglesequal} \\
\angle C_tB_tA_t &\equiv \angle D_tE_tF_t \mod 2\pi \label{eq:horizontalanglesequal} \\
\angle A_{t+1}D_{t+1}G_{t+1} &\equiv - \angle A_tD_tG_t \mod \pi \label{eq:verticalanglesconserved} \\
\angle C_{t+1}B_{t+1}A_{t+1} &\equiv -\angle C_tB_tA_t \mod \pi. \label{eq:horizontalanglesconserved}
\end{align}
\end{lemma}

\begin{proof}
Since the first two statements hold for any pattern in $\calS_{2,2}^0$, we can drop the indices $t$ to prove them. Recall that a quadrilateral is cyclic if and only if opposite angles sum to $\pi$. Denoting by $D'$ the point such that $\overrightarrow{DD'}=\vec{v}$, we obtain the following equalities modulo $2\pi$ :
\begin{align*}
\angle HEB &= \angle HED + \angle DEB \\
&= \pi - \angle DGH +\pi - \angle BAD \\
&= 2\pi - \angle DGD' \\
&= \angle D'GD \\
&= \angle ADG,
\end{align*}
where the last equality holds true because $\overrightarrow{GD'}=\overrightarrow{AD}$. Equation~\eqref{eq:horizontalanglesequal} follows similarly.

We now prove the last two statements of the lemma, which follow from the conserved quantities defined in Subsection~\ref{subsec:anglesconserved}. Define $\overrightarrow{e_1}$ (resp. $\overrightarrow{e_2}$) to be the dual edge from $\left(\tfrac{1}{2},\tfrac{1}{2}\right)$ to $\left(\tfrac{1}{2},\tfrac{3}{2}\right)$ (resp. from $\left(\tfrac{1}{2},\tfrac{3}{2}\right)$ to $\left(\tfrac{1}{2},\tfrac{5}{2}\right)$). On the torus, $(\overrightarrow{e_1},\overrightarrow{e_2})$ form a loop $l$ verifying
\[
\tilde{\gamma}_{S}(l)=\theta_{S}(e_1)-\theta_{S}(e_2).
\]
Writing ${\Omega_1}'=\Omega_1+\vec{v}$ and $E'=E+\vec{v}$, we have the following equalities modulo $2\pi$ :
\begin{align*}
\theta_{S}(e_1)-\theta_{S}(e_2)
&= \angle \Omega_2 E \Omega_1 - \angle {\Omega_1}' H \Omega_2 \\
&=\angle \Omega_2 E \Omega_1 + \angle \Omega_2 H E +  \angle E H E'  + \angle E' H {\Omega_1}'  \\
&=\angle \Omega_2 E \Omega_1 + \angle \Omega_2 H E +  \angle E H E'  + \angle E B  \Omega_1 \\
&=\angle \Omega_2 E \Omega_1 + \angle H E \Omega_2 +  \angle E H E'  + \angle \Omega_1 E B   \\
&= \angle H E B + \angle E H E' \\
&= 2\angle H E B \\
&= 2\angle A D G.
\end{align*}
By Theorem~\ref{thm:gammaconserved} we have
\[
\theta_{S_t}(e_1)-\theta_{S_t}(e_2)=-\left(\theta_{S_{t+1}}(e_1)-\theta_{S_{t+1}}(e_2)\right),
\]
hence formula~\eqref{eq:verticalanglesconserved} follows. Formula~\eqref{eq:horizontalanglesconserved} is proved similarly.
\end{proof}
\begin{remark}
\label{rem:BDtrajectory}
Lemma~\ref{lem:22angles} implies in particular that the motion of $B_{2t}$ for $t\in\Z$ is on the circle $\calC_B$ defined as the locus of the points $M$ such that $\angle C M A \equiv \angle C B A \mod \pi$. The motion of $B_{2t+1}$ is on the circle $\calC_B'$ obtained by reflecting $\calC_B$ across the line $(AC)$. Similarly, the motion of $D_{2t}$ (resp. $D_{2t+1}$) is on some circle $\calC_D$ (resp. $\calC_D'$).
\end{remark}

We will now introduce a dichotomy governing the evolution of Miquel dynamics in the $2\times2$ case, depending on whether both $\angle C B A$ and $\angle A D G$ are different from $0$ modulo $\pi$ (generic case) or not. It will imply something about the shape of the fundamental domain $A C I G$, which is always a parallelogram.

\begin{proposition}
\label{prop:dichotomy}
The fundamental domain $A C I G$ is a rectangle if and only if $\angle C B A\equiv 0\mod \pi$ or $\angle A D G \equiv 0 \mod \pi$. In that case, all the faces are trapezoids.
\end{proposition}

\begin{proof}
Assume for example that $\angle C B A \equiv 0 \mod \pi$. Then by Lemma~\ref{lem:22angles}, $B\in\left(AC\right)$, $E\in\left(DF\right)$ and $H\in\left(GI\right)$. Since all three lines are parallel, we deduce that all the faces are cyclic trapezoids. Considering the faces $A B E D$ and $D E H G$, this implies that the perpendicular bisectors of the segments $\left[AB\right]$ and $\left[GH\right]$ are both equal to the perpendicular bisector of $\left[DE\right]$. The quadrilateral $A B H G$ being a parallelogram, we conclude that $\left(A B\right)\perp\left(A G\right)$, hence $A C I G$ is a rectangle.

Conversely, assume that the fundamental domain is a rectangle. This implies that $(BH)\perp(DF)$, and by Proposition~\ref{prop:hyperbola}, this means that $E\in(BH)\cup(DF)$. Thus, by~\eqref{eq:verticalanglesequal} and~\eqref{eq:horizontalanglesequal}, $\angle C B A\equiv 0\mod \pi$ or $\angle A D G \equiv 0 \mod \pi$.
\end{proof}

\subsection{Generic case}

In this subsection, we assume that the fundamental domain is not a rectangle. We will show that the trajectory of the point $E_t$ follows some quartic curve. Fix the points $A,C,G,I$ in such a way that $ACIG$ is a parallelogram but not a rectangle (i.e. $AI$ and $CG$ have different lengths). Fix also $\alpha=\angle CBA$ and $\beta=\angle ADG$ to be two angles defined modulo $\pi$. The points $A,C,G,I$ and the angles $\alpha$ and $\beta$ are conserved quantities for the sequence $(\tilde{S}_t)_{t\in\Z}$. Where can the point $E$ lie once we prescribe these quantities ?

Set $O$ to be the common midpoint of the diagonals $\left[AI\right]$ and $\left[CG\right]$. Construct $C_1$ (resp. $C_2$) to be the center of the circle defined as the locus of the points $M$ such that $\angle CMA \equiv \alpha \mod \pi$ (resp. $\angle AMG \equiv \beta \mod \pi$). Set $C_1'=C_1+\overrightarrow{AG}$ and $C_2'=C_2+\overrightarrow{AC}$. Define the point $P$ to be the intersection of the lines $(C_1C_1')$ and $(C_2C_2')$. Let $P'$ be the reflection of $P$ across $O$. Then we have the following result regarding the location of $E$ :

\begin{theorem}
\label{thm:genericquartic}
Define $\calQ$ to be the following quartic curve :
\begin{equation}
\label{eq:genericquartic}
\calQ=\left\{M\in\R^2|PM^2P'M^2-\lambda OM^2 =k\right\},
\end{equation}
where $\lambda$ and $k$ are chosen such that $A$ and $C$ lie on $\calQ$, i.e.
\begin{align}
\lambda&=\frac{PA^2P'A^2-PC^2P'C^2}{OA^2-OC^2} \label{eq:lambda} \\
k&=\frac{OA^2PC^2P'C^2-OC^2PA^2P'A^2}{OA^2-OC^2} \label{eq:k}.
\end{align}
The above construction can be extended to a circle pattern $S\in\calS_{2,2}^0$ such that $\angle CBA \equiv \alpha \mod \pi $ and $\angle ADG \equiv\beta \mod \pi$ if and only if $E$ lies on $\calQ$.
\end{theorem}

\begin{figure}[htpb]
\centering
\includegraphics[height=3in]{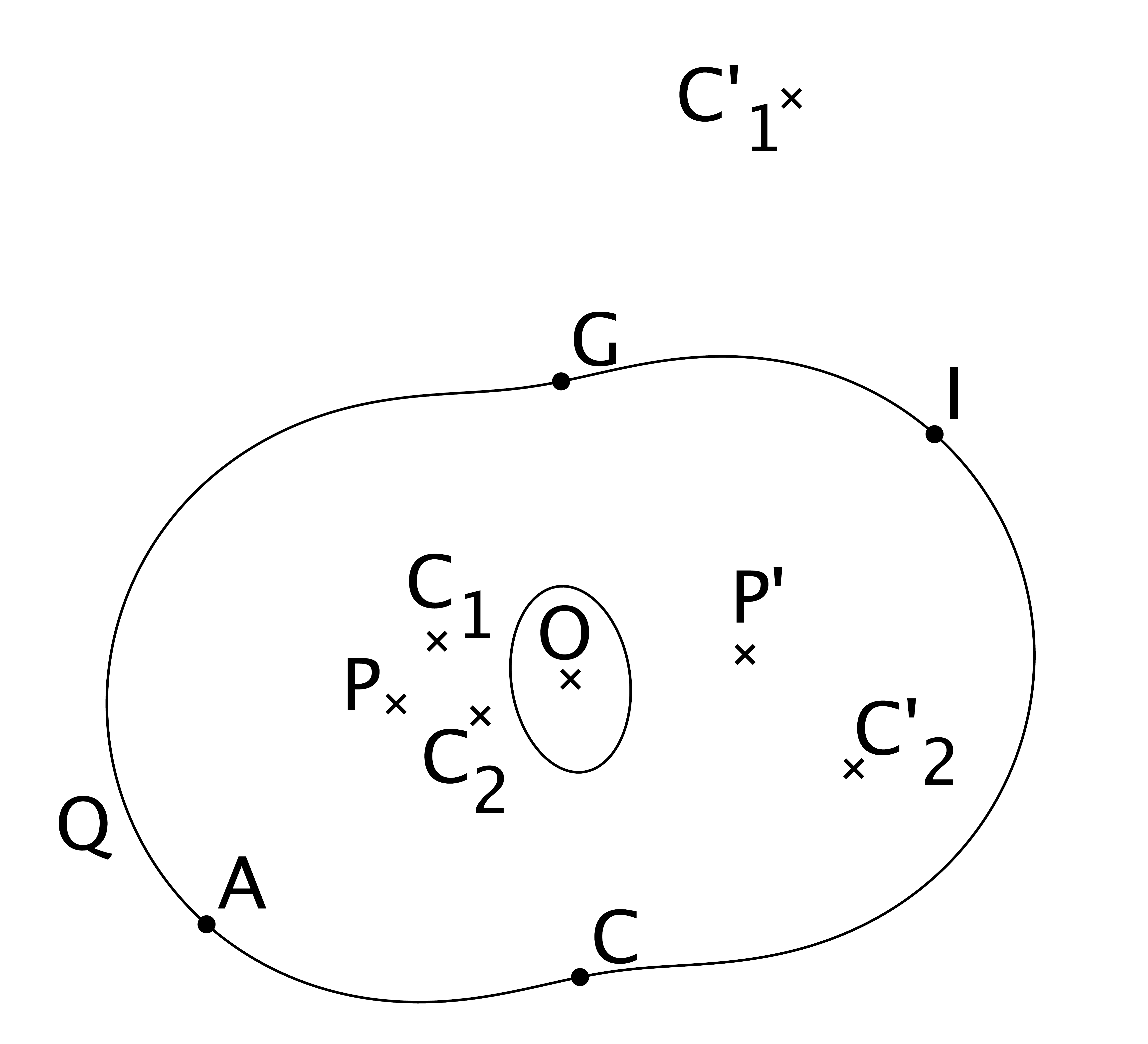}
\caption{The quartic curve $\calQ$ when the fundamental domain is not rectangular. In this example, the curve has two real ovals, but it may also have a single oval or two non-nested ovals.}
\label{fig:genericquartic}
\end{figure}

\begin{proof}
Assume we have a circle pattern $S\in\calS_{2,2}^0$. By Proposition~\ref{prop:hyperbola}, the points $B,D,E,F$ and $H$ lie on a common equilateral hyperbola $\calH$. We apply to $S$ the similarity mapping $\calH$ to the hyperbola of equation $xy=1$. The coordinates of $B,D,E,F$ and $H$ are now all of the form $\left(x,\tfrac{1}{x}\right)$. Using a computer algebra program, we compute the coordinates of $\Omega_1$ and $\Omega_2$, deduce from it the coordinates of $A$, $C$, $G$ and $I$, then the coordinates of $C_1$ and $C_2$. From these, we obtain the coordinates of the points $O$, $P$ and $P'$. Defining $\lambda$ and $k$ using equations~\eqref{eq:lambda} and~\eqref{eq:k}, we can finally show that
\[
PE^2P'E^2-\lambda OE^2 -k=0,
\]
thus $E$ lies on $\calQ$.

Conversely, pick $E$ on $\calQ$. In order to extend the points $A,C,E,G,I$ to a circle pattern verifying $\angle CBA \equiv \alpha \mod \pi$ and $\angle ADG \equiv \beta \mod \pi$, we need to find a point $D$ satisfying the following three conditions, each of them expressed by an equation involving the Cartesian coordinates of the points :
\begin{enumerate}
\item $D$ lies on the circle of center $C_2$ going through $A$ (equation 1)
\item $D$ lies on the circle of center $C_1+\overrightarrow{CE}$ going through $E$ (equation 2)
\item $D$ lies on the equilateral hyperbola going through $A,E,G$ and $E+\overrightarrow{CA}$ (equation 3).
\end{enumerate}
In these equations, the coordinates of $A,C,G,C_1$ and $C_2$ are fixed and the variables are $x_D,y_D,x_E$ and $y_E$. Taking successive resultants with a computer algebra program, we eliminate the variables $x_D$ and $y_D$ and obtain a polynomial equation involving only $x_E$ and $y_E$. The vanishing locus of the corresponding polynomial is the union of four irreducible components, one of them being the quartic curve $\calQ$. Since $E$ lies on $\calQ$, we can find a point $D$ such that equations 1,2 and 3 are satisfied. Thus $D$ lies on the equilateral hyperbola going through $A,E,G$ and $E+\overrightarrow{CA}$, and by (a shifted version of) Proposition~\ref{prop:hyperbola}, we have that the points $A,C,D,E,G,I$ can be extended to a circle pattern in $\calS_{2,2}^0$. By equation 1 (resp. 2), we have $\angle ADG=\beta \mod \pi$ (resp. $\angle EDE'=\alpha \mod \pi$, where $E'=E+\overrightarrow{CA}$). By Lemma~\ref{lem:22angles} we have $\angle CBA = \angle EDE'$, hence the conclusion.
\end{proof}

We deduce from Theorem~\ref{thm:genericquartic} the following consequence :

\begin{corollary}
When the fundamental domain is not a rectangle, the trajectory of $E_t$ is along the quartic curve $\calQ$ defined above.
\end{corollary}

\begin{proof}
Denote by $\calQ_{A,C,G,I}^{\alpha, \beta}$ the quartic curve constructed from the points $A,C,G$ and $I$ and $\alpha,\beta$ the angles modulo $\pi$. Firstly, observe that the position of $A_t$, $C_t$, $G_t$ and $I_t$ is independent of $t$. Secondly, if we write $\alpha_0=\angle C_0B_0A_0$ and $\beta_0=\angle A_0 D_0 G_0$, by Lemma~\ref{lem:22angles}, we have $\angle C_{t}B_{t}A_{t} \equiv (-1)^t \alpha_0 \mod \pi $ and $\angle A_{t}D_{t}G_{t} \equiv (-1)^t \beta_0 \mod \pi$. Thus the quartic curve constructed from the pattern $\tilde{S}_t$ (which contains the point $E_t$) is either $\calQ=\calQ_{A_0,C_0,G_0,I_0}^{\alpha_0,\beta_0}$ or $\hat{\calQ}=\calQ_{A_0,C_0,G_0,I_0}^{-\alpha_0,-\beta_0}$. Denoting by $\hat{C}_1,\hat{C}_2,\hat{C}_1',\hat{C}_2',\hat{P}$ and $\hat{P'}$ the counterparts of $C_1,C_2,C_1',C_2',P$ and $P'$ entering the construction of $\hat{\calQ}$, we observe that $\hat{C}_1,\hat{C}_2,\hat{C}_1'$ and $\hat{C}_2'$ are the respective images of $C_1,C_2,C_1'$ and $C_2'$ under the reflection across $O$. Thus $\hat{P}=P'$ and $\hat{P'}=P$, and we conclude that $\calQ=\hat{\calQ}$.
\end{proof}

\begin{remark}
The curve $\calQ$ can have one or two real ovals. In the case when it has two nested ovals, that $E_0$ lies on the inner oval and $A_0,C_0,G_0,I_0$ lie on the outer oval, as on Figure~\ref{fig:genericquartic}, we observe experimentally that the trajectory of $E_t$ stays on this inner oval. See Figure~\ref{fig:22integrability} for a  plot of several iterates of $E_t$.
\end{remark}

\begin{figure}[htpb]
\centering
\includegraphics[height=2in]{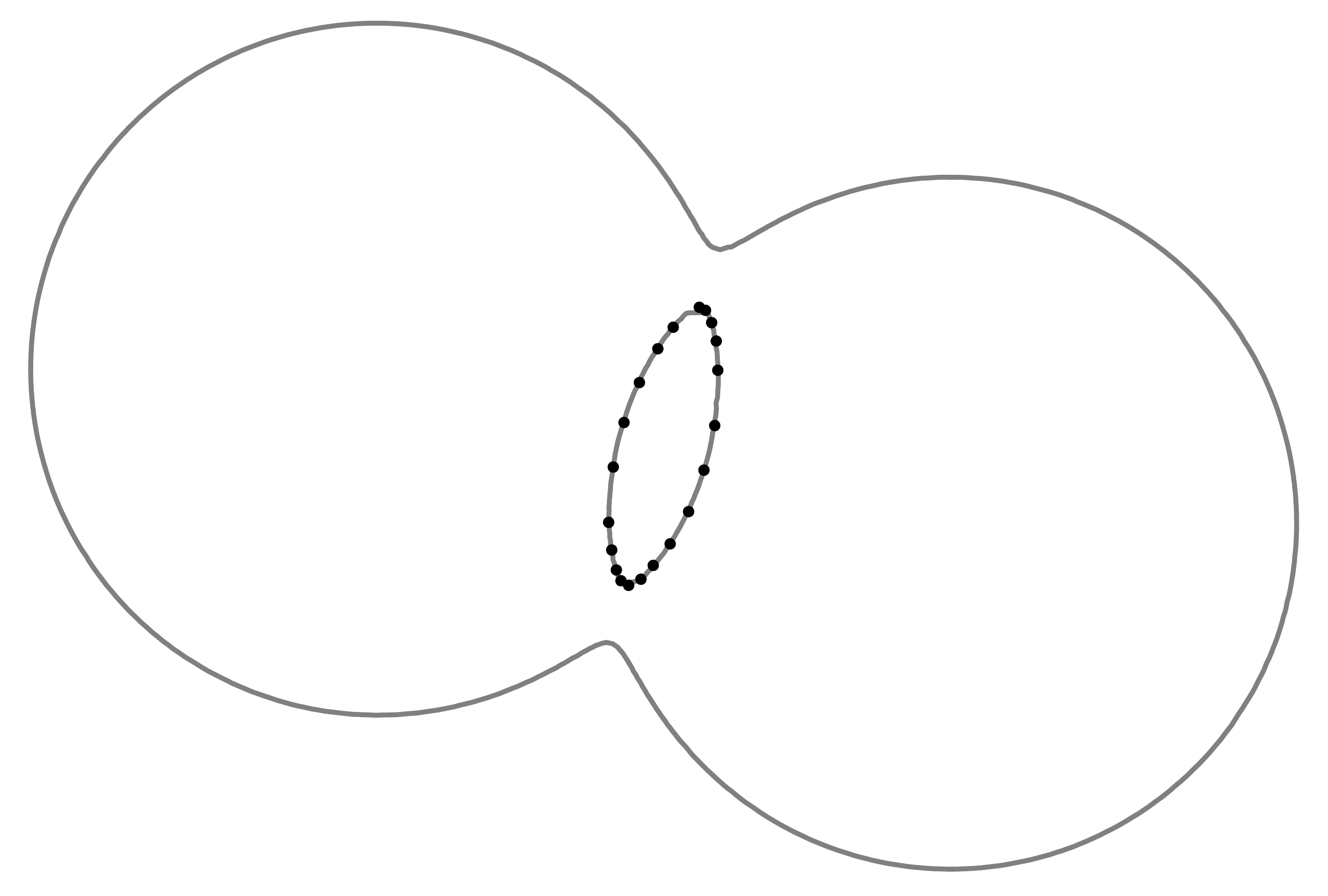}
\caption{Computer simulation of the motion of $E_t$ on the quartic curve for ten iterations of each mutation.}
\label{fig:22integrability}
\end{figure}

\subsection{Rectangular fundamental domain}

In this subsection, we consider the case when the fundamental domain is a rectangle. By Proposition~\ref{prop:dichotomy}, at least one of the angles $\angle CBA$ or $\angle ADG$ is flat. Without loss of generality, we assume that $\angle CBA \equiv 0 \mod \pi$.

We assign coordinates to the vertices of the circle pattern $S$ by picking the center of the coordinates to be the center of the rectangle $ACIG$ and the horizontal axis to be in the direction of  $\overrightarrow{AC}$. Then we have the following coordinates :
\begin{align*}
&A=\left(-x_I,-y_I\right) &&C=\left(x_I,-y_I\right) &&&G=\left(-x_I,y_I\right) \\
&I=\left(x_I,y_I\right) && D=\left(x_D,y_E\right) &&& E=\left(x_E,y_E\right).
\end{align*}
Define the quantities
\begin{align}
a&=x_I^2+y_I^2+x_E^2+y_E^2+\frac{\left(x_D+x_I\right)^2\left(x_I^2+y_I^2-x_E^2-y_E^2\right)}{y_I^2-y_E^2} \\
b&=x_I^2+y_I^2+x_E^2+y_E^2+\frac{\left(x_D+x_I\right)^2\left(x_E^2-x_I^2\right)\left(x_I^2+y_I^2-x_E^2-y_E^2\right)}{\left(y_I^2-y_E^2\right)^2} \\
c&=\left(x_I^2+y_I^2\right)\left(x_E^2+y_E^2\right)+\frac{\left(x_D+x_I\right)^2\left(x_E^2y_I^2-x_I^2y_E^2\right)\left(x_I^2+y_I^2-x_E^2-y_E^2\right)}{\left(y_I^2-y_E^2\right)^2}.
\end{align}
Define the quartic curve $\calK_S$ to be the real vanishing locus of the polynomial $(X^2+Y^2)^2-aX^2-bY^2+c$. This curve supports the trajectory of $E_t$ :

\begin{theorem}
\label{thm:trapezoidalquartic}
For any $t\in\Z$, the point $E_t$ lies on the quartic curve $\calK_{S}$. Moreover, the points $A,C,G$ and $I$ also lie on $\calK_{S}$.
\end{theorem}

\begin{figure}[htpb]
\centering
\includegraphics[height=2.5in]{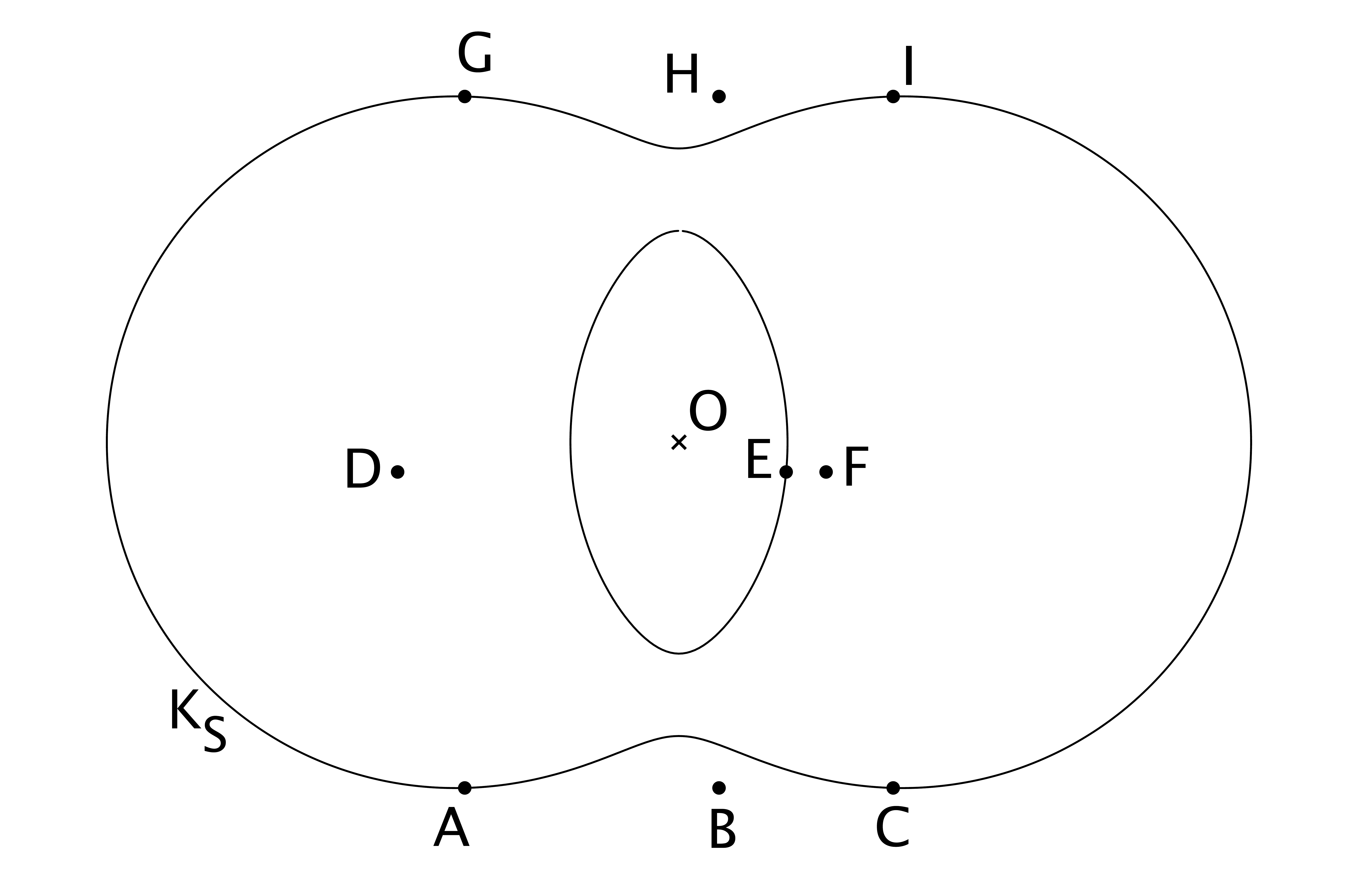}
\caption{The quartic curve $\calK_S$ when the fundamental domain is rectangular. In this example, the curve has two nested real ovals, but it may also have a single oval or two non-nested ovals.}
\label{fig:trapezoidalquartic}
\end{figure}

\begin{proof}
One can check that the equation defining $\calK_S$ can be rewritten as
\begin{multline}
\label{eq:determinantquartic}
\left(y_I^2-y_E^2\right)^2\left(X^2+Y^2-x_I^2-y_I^2\right)\left(X^2+Y^2-x_E^2-y_E^2\right)=\\
\left(x_D+x_I\right)^2\left(x_E^2+y_E^2-x_I^2-y_I^2\right)
\begin{vmatrix}
1 & x_E^2 & y_E^2 \\
1 & x_I^2 & y_I^2 \\
1 & X^2 & Y^2
\end{vmatrix}.
\end{multline}
From this formula, it is clear that the points $A,C,G,I$ and $E$ lie on $\calK_S$. Denote by $a_t,b_t$ and $c_t$ the time-dependent versions of $a,b$ and $c$ (where $D,E$ and $I$ are replaced by $D_t,E_t$ and $I_t$). Using a computer algebra program, one can compute the coordinates of $D_1,E_1,D_{-1}$ and $E_{-1}$ as functions of the coordinates of $D_0,E_0$ and $I_0$ and deduce that
\begin{align}
a_1=a_{-1}&=a_0 \\
b_1=b_{-1}&=b_0 \\
c_1=c_{-1}&=c_0.
\end{align}
This implies $\calK_{S_1}=\calK_{S_{-1}}=\calK_{S_0}$ and by induction, $\calK_{S_t}=\calK_{S_0}$ for any $t\in\Z$. Thus $E_t\in\calK_{S_0}$ for any $t\in\Z$.
\end{proof}

Here, when the angle $\angle CBA$ is flat, all the faces are trapezoids. When we additionally impose that $\angle ADG$ is flat, all the faces become rectangles. Then, since $x_D=-x_I$, by formula~\eqref{eq:determinantquartic}, the quartic $\calK_S$ degenerates to the union of two circles centered at the center of the fundamental domain, one going through $E$ and the other one going through $A,D,G$ and $I$. In this case $E_t$ stays on the circle containing $E_0$ and we can even compute the speed of rotation. Write $\delta=\angle COI$, where $O$ is the center of the fundamental domain. Let $\sigma$ denote the reflection across the horizontal axis (going through $O$ and parallel to $(AC)$) and for any angle $\gamma$, let $r_\gamma$  denote the rotation of center $O$ and angle $\gamma$. Then the following holds :

\begin{theorem}
\label{thm:rectangularrotation}
For any $t\in\Z$,
\begin{align}
E_{2t+1}&=\sigma\circ r_{-\delta}\left(E_{2t}\right) \\
E_{2t}&=\sigma\circ r_\delta\left(E_{2t-1}\right).
\end{align}
In particular, for any $t\in\Z$,
\begin{align}
E_{2t+1}&=r_{2\delta}\left(E_{2t-1}\right) \\
E_{2t+2}&=r_{-2\delta}\left(E_{2t}\right).
\end{align}
\end{theorem}

\begin{proof}
This follows from computing with a computer algebra program the coordinates of $E_{t+1}$ as functions of the coordinates of $E_t$ and $I_t$ (distinguishing when $t$ is even or odd) and observing that $\cos \delta = \frac{x_I^2-y_I^2}{x_I^2+y_I^2}$ and $\sin \delta = \frac{2x_Iy_I}{x_I^2+y_I^2}$.
\end{proof}

\paragraph*{Acknowledgements}

I thank Richard Kenyon for suggesting the study of this system and for his advice throughout the course of this project. I am also grateful to an anonymous referee for noticing a simplification of the proof of Proposition~\ref{prop:effective}. Finally, I acknowledge the hospitality of the Institute for Computational and Experimental Research in Mathematics (ICERM) in Providence and the Statistical Laboratory in Cambridge, where part of this work was completed.

\label{Bibliography}
\bibliographystyle{plain}
\bibliography{bibliographie}

\Addresses
\end{document}